\documentclass[reqno,12pt, centertags]{amsart}
\usepackage[left=1in,right=1in,bottom=1in,top=1in]{geometry}
\usepackage{amsmath,amsthm,amscd,amssymb}
\usepackage{amsmath,amsthm,amscd,amssymb,color,eucal,enumitem,latexsym,lscape,graphicx,multirow,mathrsfs,mathtools}
\usepackage{upref}
\usepackage{latexsym}
\usepackage{lscape}
\usepackage{graphicx}
\usepackage{multirow}
\usepackage{mathrsfs} 
\usepackage{multicol}
\usepackage{mathtools}
\usepackage{subcaption}
\mathtoolsset{showonlyrefs}

\newcommand{\bbN}{{\mathbb{N}}}
\newcommand{\bbR}{{\mathbb{R}}}

\newcommand{\bbC}{{\mathbb{C}}}

\newcommand{\cA}{{\mathcal A}}
\newcommand{\cB}{{\mathcal B}}

\newcommand{\cD}{{\mathcal D}}
\newcommand{\cE}{{\mathcal E}}
\newcommand{\cF}{{\mathcal F}}
\newcommand{\cG}{{\Gamma}}

\newcommand{\cS}{{\mathcal S}}
\newcommand{\cT}{{\mathcal T}}

\newcommand{\cV}{{\mathcal V}}

\newcommand{\cX}{{\mathcal X}}

\newcommand{\lb}{\label}

\newcommand{\ran}{\text{\rm{ran}}}

\newcommand{\hatt}{\widehat}

\newcommand{\elpee}{L^p(\Gamma)}

\newcommand{\wpee}{\hatt W^{1, p}(\Gamma)}
\newcommand{\elq}{L^q(\Gamma)}

\newcommand{\el}[1]{L^{#1}(\Gamma)}

\numberwithin{equation}{section}

\renewcommand{\det}{\operatorname{det}}

\newcommand{\dom}{\operatorname{dom}}

\renewcommand{\Re}{\operatorname{Re }}

\renewcommand{\ker}{\operatorname{ker}}

\theoremstyle{plain}
\newtheorem{theorem}{Theorem}[section]

\newtheorem{lemma}{Lemma}[section]

\newtheorem{proposition}{Proposition}[section]

\theoremstyle{definition}

\newcommand{\blue}{\color{blue}}

\DeclareMathOperator{\spec}{Spec}

\begin{document}

\allowdisplaybreaks

\numberwithin{equation}{section}
\allowdisplaybreaks

\title[Logistic Keller-Segel models ]{Stability and bifurcation for logistic Keller--Segel models on compact graphs}

\author[H. Shemtaga]{Hewan Shemtaga}
\address{Department of Mathematics and Statistics,
	Auburn University, Auburn, AL 36849, USA}
\email{hms0069@auburn.edu}

\author[W.\ Shen]{Wenxian Shen}
\address{Department of Mathematics and Statistics,
	Auburn University, Auburn, AL 36849, USA}
\email{wenxish@auburn.edu}

\author[S.\ Sukhtaiev]{Selim Sukhtaiev	}
\address{Department of Mathematics and Statistics,
	Auburn University, Auburn, AL 36849, USA}
\email{szs0266@auburn.edu}
\thanks {We thank Gregory Berkolaiko for sharing his MATLAB scripts \cite{GBhomepage} and for numerous discussions. HS and SS were supported in part by NSF grant DMS-2243027, Simons Foundation grant MP-TSM-00002897, and by the Office of the Vice President for Research \& Economic Development (OVPRED) at Auburn University through the Research Support Program grant.}

\dedicatory{To Fritz Gesztesy on the occasion of his $70^{th}$ birthday}

\date{\today}
\keywords{}

\begin{abstract}
This paper concerns asymptotic stability,  instability, and bifurcation  of constant steady state solutions of the parabolic-parabolic and parabolic-elliptic chemotaxis models on metric graphs. We determine a threshold value $\chi^*>0$ of the  chemotaxis sensitivity parameter that separates the regimes of local asymptotic stability and instability, and, in addition, determine the parameter intervals that facilitate global asymptotic convergence of solutions with positive initial data to constant steady states. Moreover, we provide a sequence of bifurcation points for the chemotaxis sensitivity parameter that yields non-constant steady state solutions. In particular, we show that the first bifurcation point coincides with threshold value $\chi^*$ for a generic compact metric graph. Finally, we supply numerical computation of bifurcation points for several graphs. 
\end{abstract}

\maketitle

{\scriptsize{\tableofcontents}}


\section{Introduction}
This paper is centered around the Keller--Segel model given by the following initial value problem for a system of reaction-advection-diffusion equations on a metric graph $\Gamma=(\mathcal{V},\mathcal{E})$
\begin{equation}
	\label{parabolic-parabolic-eq}
	\begin{cases}
		u_t=\partial_x\big(\partial_{x} u-\chi u\partial_xv\big)+u(a-bu), \quad { x\in\mathcal{E}}, \cr
		\tau v_t=\partial_{xx}^2 v-v+u, \quad { x\in\mathcal{E}},\cr
		u(0, x)=u_0(x),\ v(0,x)=v_0(x),\quad {x\in \mathcal{E}},
	\end{cases}
\end{equation}
where $\chi, a,b>0, \tau\ge 0$,  $\cV$ is the set of vertices and $\cE$ is the set of edges of the graph.
This pair of PDEs describes population dynamics in presence of attracting substances. In the context of the directed movement of microorganisms in response to a chemical attractant this model is often referred to as the chemotaxis model, see \cite{KJP} for illuminating discussion of chemotaxis phenomena in biomedical and social  sciences. The vast mathematical literature on this model  includes \cite{MR3351175, MR3698165, MR2409228, MR2448428, MR2013508,  IsSh1, IsSh2, MR3925816, KS1, KoWeXu, MR3294344, MR3620027, MR3397319, MR2334836, MR1654389, MR3147229, WaYaGa, MR2445771, MR2644137, MR2754053, MR2825180, MR3115832, MR3210023, MR3462549, MR3335922, MR3286576}.  

The two quantities central to chemotaxis models are the population density $u=u(t,x)$ and the concentration of the chemical substance $v=v(t,x)$. The classical Fick's law of diffusion combined with  population drifts along the chemical gradient yield the first equation in \eqref{parabolic-parabolic-eq}, where $\chi>0$ is the chemotaxis sensitivity parameter and $\chi u\partial_x v$ is the taxis-flux term. The second equation in \eqref{parabolic-parabolic-eq} is the usual reaction-diffusion equation, with $\tau>0$ corresponding to moderate diffusion rate of the chemical substances and $\tau=0$ corresponding to rapid diffusion thereof. 

In this paper, we consider \eqref{parabolic-parabolic-eq} in two regimes:
\begin{enumerate}
	\item $\tau=0$, in which case \eqref{parabolic-parabolic-eq} is referred to as the parabolic-elliptic system,
	\item $\tau>0$, in which case \eqref{parabolic-parabolic-eq} is referred to as the parabolic-parabolic system.
\end{enumerate}
Treating both regimes by different methods, we focus on local asymptotic stability of the constant steady state $(u_0, v_0)=(a/b, a/b)$ and existence of non-trivial steady states bifurcating from $(u_0, v_0)$ in response to small variation of the chemotaxis sensitivity parameter $\chi$.  We investigate \eqref{parabolic-parabolic-eq} posed on arbitrary connected compact metric graphs $\Gamma={(\mathcal{V},\mathcal{E})}$ and consider solutions $u=u(t,x)$, $v=v(t,x)$ satisfying natural Neumann--Kirchhoff vertex conditions describing continuity and preservation of flux at all vertices $\vartheta\in\cV$, that is, 
\begin{equation}
	\label{NKsol}
	\begin{cases}
		u_e(\vartheta)=u_{e'}(\vartheta), v_e(\vartheta)=v_{e'}(\vartheta), \,\, e\sim \vartheta, e'\sim\vartheta, \,\, { \vartheta\in\mathcal{V}}\text{ (continuity at vertices)},\cr
		\sum\limits_{\vartheta \sim e} \partial_{\nu}u_e(\vartheta)=0\sum\limits_{\vartheta \sim e} \partial_{\nu}v_e(\vartheta)=0,\,\, { \vartheta\in\mathcal{V}}. \text{ (conservation of current),}
	\end{cases}
\end{equation}
where $\partial_{\nu}u_e(\vartheta)$ denotes the inward normal derivative of $u$ along the edge $e$ at the vertex $\vartheta$. 

{In \cite{HWS}, we have established well-posedness for  general chemotaxis systems on arbitrary compact metric graphs including \eqref{parabolic-parabolic-eq}{, \eqref{NKsol}} as a special case, see Theorem \ref{Lp}. In this paper, we investigate the stability, instability, and bifurcation of the constant solution $(\frac{a}{b},\frac{a}{b})$ of \eqref{parabolic-parabolic-eq}{, \eqref{NKsol}}. 
}

Our first result  provides a  threshold value $\chi^*>0$ of the  chemotaxis sensitivity parameter that separates the regimes of local asymptotic stability and instability
of the constant solution $(u_0, v_0)$.
\begin{theorem}[Local asymptotic stability and instability]\label{local-stability-thm} Let $\Gamma$ be a connected compact metric graph and let 
$\chi(\lambda)$, $\chi^*\in(0,\infty)$ be defined by
\begin{equation}
\label{chi-lambda-eq}
{\chi(\lambda):=\frac{b(\lambda-a)(1-\lambda)}{a\lambda}}, \lambda<0,
\end{equation}
and 
		\begin{align}
		&\chi^*:=\min\left \{{\chi(\lambda)}: \lambda\in\spec(\Delta)\setminus\{0\} \right\},	
	\end{align}
{where $\spec(\Delta)$  is the spectrum of   the Neumann--Kirchhoff Laplacian
acting in $L^2(\Gamma)$.}
Then the following assertions hold for $\tau\geq 0$.
\begin{enumerate}
\item If $0<\chi<\chi^*$ then the constant solution $(\frac{a}{b},\frac{a}{b})$ of \eqref{parabolic-parabolic-eq}{, \eqref{NKsol}} is locally asymptotically stable.

\item If $\chi>\chi^*$  then the constant solution $(\frac{a}{b},\frac{a}{b})$ of \eqref{parabolic-parabolic-eq}{, \eqref{NKsol}}  is unstable.
\end{enumerate}
\end{theorem}

To prove Theorem \ref{local-stability-thm} we compute the spectrum, via finding zeros of the perturbation determinant, of the linearization of \eqref{parabolic-parabolic-eq} or, equivalently, of 
\begin{equation}
	\begin{bmatrix}
		1& 0\\
		0&\tau
	\end{bmatrix}\begin{bmatrix}
		\partial_tu\\ \partial_t v
	\end{bmatrix}={ \mathcal{H}(u,v,  \chi)}:=\begin{bmatrix}
		\partial_x\big(\partial_{x} u-\chi u\partial_xv\big)+u(a-bu) \\ 
\partial_{xx}^2 v-v+u
	\end{bmatrix},
\end{equation}
about the constant steady state, see Lemma \ref{eigenvalue-lm2}.

Our next result stems from a simple observation  that $\mathcal{H}(a/b,a/b, \chi)=0$ for 
all  $\chi\geq 0$. That is, for both parabolic-parabolic and parabolic-elliptic systems $(a/b, a/b, \chi)$ is the line of constant solutions in the space $(u,v,\chi)\in \hatt W^{2,2}({\Gamma})\times \hatt W^{2,2}({\Gamma})\times (0,\infty)$ {(see \eqref{L-p-eq} in Appendix A for the definition of $\hatt W^{2,2}(\Gamma)$ and other functional spaces on graphs)}. We show that the eigenvalues of the Neumann--Kirchhoff Laplacian give rise to a sequence $\{\chi_n\}_{n\geq 1}$ of bifurcation points that is bounded from below and that accumulates only at $+\infty$. Importantly, the first bifurcation point is precisely the threshold value $\chi^*$ where stability of the constant steady state ceases to take place.

\begin{theorem}\lb{prop4.4}
	Let $\Gamma$ be a connected compact metric graph. Let $\lambda\in\spec(\Delta)$ be a simple eigenvalue of the Neumann--Kirchhoff Laplacian on $\Gamma$, let $\varphi$ be the corresponding eigenfunction and define
	\begin{align}
&\label{D-eq}
		\cD:=\left\{u\in \hatt W^{2,2}(\Gamma):  \sum\limits_{\vartheta \sim e} \partial_{\nu}u_e(\vartheta)=0,\ 
		u_e(\vartheta)=u_{e'}(\vartheta), \vartheta\sim e, e'\right\},\\
	&		\hspace{5cm}	\chi_{\lambda}:={\chi(\lambda)}. \lb{chieq}		
\end{align}
	Assume, in addition, that $\chi_{\lambda}\not= \chi_{\mu}$ for $\mu\in\spec(\Delta)\setminus\{\lambda\}$ then $\chi_{\lambda}$ is a bifurcation point  { of \eqref{parabolic-parabolic-eq}, \eqref{NKsol}}, that is, there exist $\varepsilon>0$ and $\chi\in C^2((-\varepsilon, \varepsilon), \bbR)$, $\Phi\in C^2((-\varepsilon, \varepsilon), \cD\times \cD)$ such that ${\mathcal{H}}\left(\Phi(s), \chi(s)\right)=0$ for $s\in (-\varepsilon, \varepsilon)$ and 
	\begin{align}
		&\chi(0)=\chi_{\lambda}, \Phi(s)=\begin{bmatrix}
			a/b\\  a/b
		\end{bmatrix}+
		s\begin{bmatrix}
			u\\ v
		\end{bmatrix}\varphi+ o(s) \text{\ in  $\hatt W^{2,2}(\Gamma)\times \hatt W^{2,2}(\Gamma)\ $as\ }s\rightarrow 0,
	\end{align}
	where  
	\begin{equation}\lb{auxmatnew}
		\begin{bmatrix}
			u\\ v
		\end{bmatrix}\in \ker M,\ \ M:=\begin{bmatrix}
			\lambda-a-\frac{\chi a}{b}& 	\frac{\chi a}{b} \\
			1&\lambda-1
		\end{bmatrix}.
	\end{equation}
	Moreover, there exists an open set $U\subset \cD\times \cD\times\bbR$ containing $\left(\frac ab, \frac ab, \chi_{\lambda}\right)$ such that
	\begin{align}
		&\left\{ (u,v, \chi)\in U: { \mathcal{H}} (u,v,\chi)=0, (u,v)\not=\left(\frac ab, \frac ab\right)\right\}=\left\{(\Phi(s), \chi(s)): |s|<\varepsilon\right\}. 
	\end{align}
\end{theorem} 
We stress that both assumptions of Theorem \ref{prop4.4} hold automatically for a generic connected graph $\Gamma$ that has no vertices of degree $2$ and that is not a circle. That is, for a given combinatorial graph $(\cV, \cE)$ that has no loops there exists a dense $G_{\delta}$ set  $\cS\subset \bbR^{|\cE|}_+$ of edge lengths such that the corresponding metric graph $\Gamma$ with edge lengths $\{\ell_e, e\in\cE\}\in \cS$ satisfies the assumptions of Theorem \ref{prop4.4}. Indeed, the eigenvalues of the Neumann--Kirchhoff Laplacian are simple for generic graph $\Gamma$, see \cite{MR2151598}. In addition, since the eigenvalues of the Neumann--Kirchhoff Laplacian depend continuously on edge lengths, the function $\spec(\Delta)\ni\lambda\mapsto \chi_{\lambda}\in(0,\infty)$ is injective up to small variation of $\{\ell_e, e\in\cE\}\in \cS$. Figure \ref{dumbbell} illustrates numerically the bifurcation points for the dumbell graph, see also Section \ref{numerics} for more numerical examples.   

{We note that in the absence of chemotaxis, that is, when $\chi=0$, the constant solution $(\frac{a}{b},\frac{a}{b})$ of \eqref{parabolic-parabolic-eq}, \eqref{NKsol} is globally stable. In particular, no non-constant steady states of \eqref{parabolic-parabolic-eq}, \eqref{NKsol} exists for $\chi=0$.
In this context, Theorem \ref{prop4.4} shows that  chemotaxis induces non-constant steady states via the bifurcation of the constant steady state. In the special case of $\Gamma$ being a single interval the bifurcation of the constant steady state and existence of spiky solutions have been investigated, for example,  in \cite{CaXuLi, KoWeXu, MR2334836, WaYaGa, WaXu}.
}
\begin{figure}
	\begin{subfigure}{.5\textwidth}
		\centering
		\includegraphics[width=1.3\linewidth]{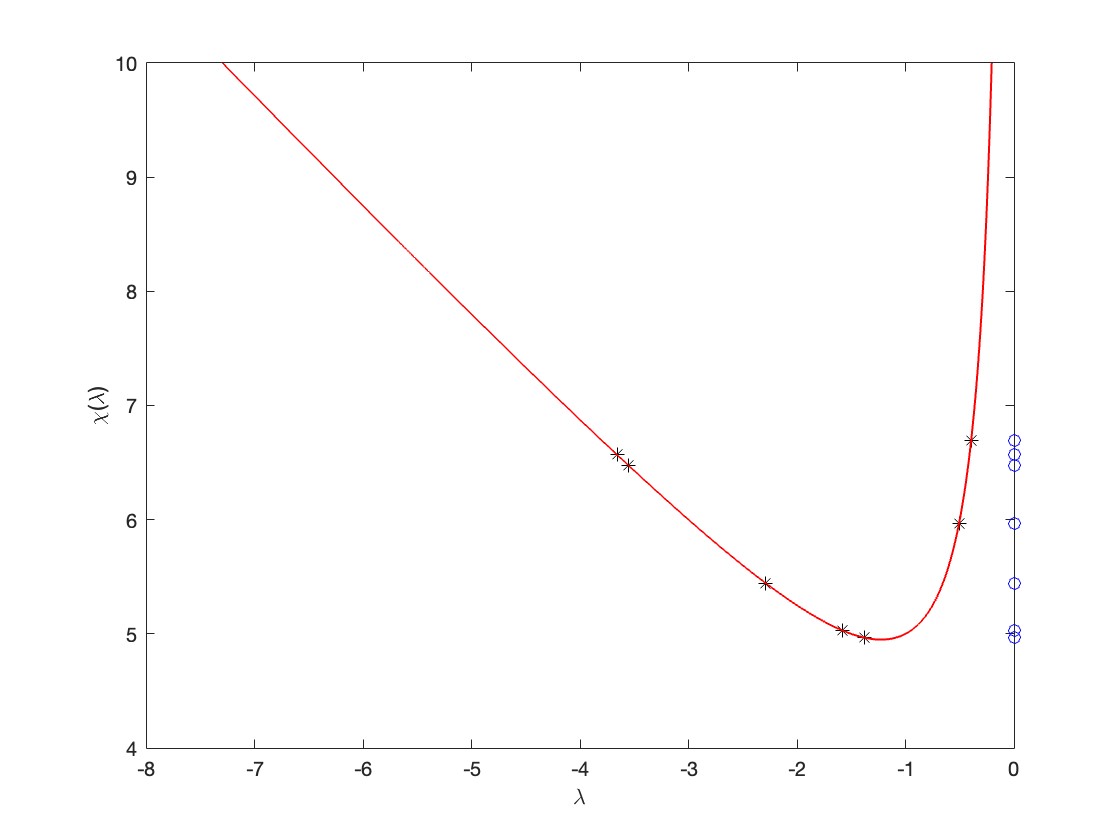}
	\end{subfigure}%
	\begin{subfigure}{.5\textwidth}
		\centering
		\includegraphics[scale=0.15]{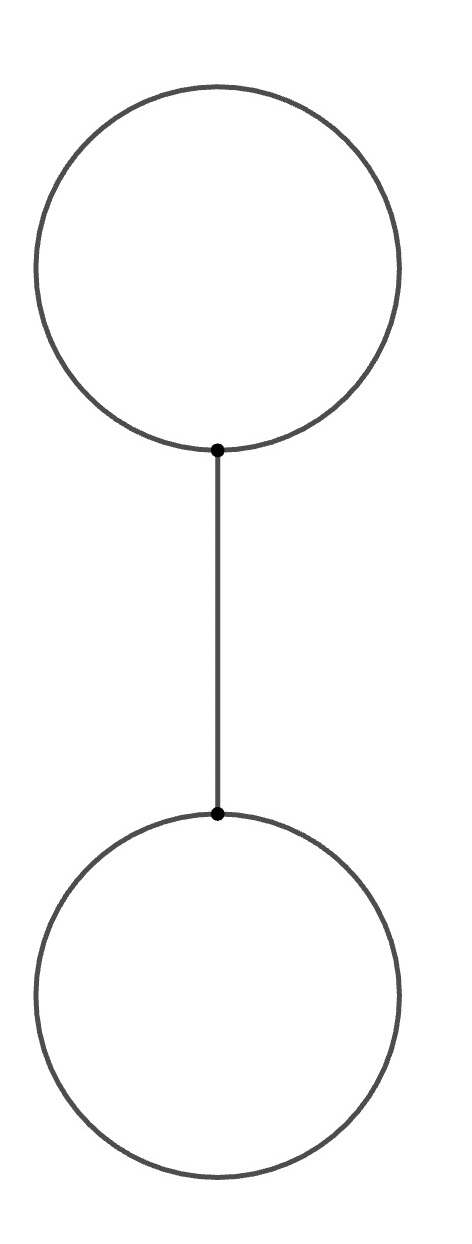}
	\end{subfigure}
	\caption{The red curve is the graph of $\chi=\chi(\lambda)$ with $a=b=1.5$; $*$ indicates the values of $\chi$ at eigenvalues of the Kirchhoff Laplacian on a {\it dumbbell graph }with edge lengths $10, 5, 1$; {\blue $\circ$} indicate bifurcation points. The first bifurcation point $\chi^*\approx4.96489$ corresponds to the $4-$th eigenvalue.}
	\label{dumbbell}
\end{figure}

The next two theorems concern global asymptotic convergence of solutions with non-trivial non-negative initial data to the constant steady state $(a/b, a/b)$ in the following regimes:
\begin{itemize}
\item for $\chi$ satisfying \begin{align}\lb{kap}
	&\frac {b^2}{\chi^2}   > \frac{\left(a+\chi \kappa(a,b,\chi) \right)^2}{a-\chi \kappa(a,b,\chi) } \text{\ and\ } a-\chi \kappa(a,b,\chi) >0,
\end{align}
where
\begin{align}
	&\kappa(a,b,\chi):={ C \left(\frac{2a^2}{b^2}+\frac{a^2\chi}{b^2}+\frac{a^2\chi^2}{b^3}+\frac{a^3\chi^3}{b^4}+\frac{a^4\chi^4}{b^5} \right)}, 
\  C=C(\Gamma)>0,
\end{align} 
in the parabolic-parabolic model, see Theorem \ref{global-stability-thm},

\item for $\chi\in(0,b/2)$ in the parabolic-elliptic model, see Theorem \ref{4.4}.
\end{itemize}

\begin{theorem} [Global stability for parabolic-parabolic model]
\label{global-stability-thm}  	Let $\Gamma$ be a connected compact metric graph. Let $u_0\in\hatt C(\overline{\Gamma})$, $v_0\in \hatt C^1(\overline{\Gamma})$ be non-negative initial data  { $u_0\not \equiv 0$} and let $u=u(x, t; u_0,v_0)$, $v=v(x, t; u_0,v_0)$ be a global unique positive solution of \eqref{parabolic-parabolic-eq} with $\tau>0$ satisfying the Neumann--Kirchhoff vertex conditions \eqref{NKsol} { and the initial condition $(u(x,0;u_0,v_0), v(x,0;u_0,v_0))=(u_0(x),v_0(x))$} \footnote{cf. Theorem \ref{Lp} (2)}. Then there exists $C=C(\Gamma)>0$ such that for $\chi>0$ satisfying  \eqref{kap} one has
\begin{equation}\lb{asympstab1}
	\lim_{t \rightarrow \infty}\left(\left\|u(t, \cdot;u_0, v_0)-\frac{a}{b}\right\|_{L^{\infty}(\Gamma)} + \left\|v(t, \cdot;u_0, v_0)-\frac{a}{b}\right\|_{L^{\infty}(\Gamma)} \right)=0. 
\end{equation}
\end{theorem}

\begin{theorem}[Global stability for parabolic-elliptic model]\lb{4.4}
	Let $\Gamma$ be a connected compact metric graph.  Let $u_0\in\hatt C(\overline{\Gamma})$ be non-negative initial data $u_0\not\equiv 0$ and let $u=u(x, t; u_0)$, $v=v(x, t; u_0)$ be a global unique positive solution of \eqref{parabolic-parabolic-eq} with $\tau=0$ satisfying the Neumann--Kirchhoff vertex conditions \eqref{NKsol} { and the initial condition $u(x,0;u_0)=u_0(x)$} \footnote{cf. Theorem \ref{Lp} (1)}. Then for $\chi\in (0, b/2)$ one has
	\begin{equation}
		\lim_{t \rightarrow \infty}\left(\left\|u(t, \cdot;u_0)-\frac{a}{b}\right\|_{L^{\infty}(\Gamma)} + \left\|v(t, \cdot;u_0)-\frac{a}{b}\right\|_{L^{\infty}(\Gamma)} \right)=0. 
	\end{equation}
\end{theorem}

{The global stability of the positive constant solution 
for \eqref{parabolic-parabolic-eq} with $\tau>0$ on  regular convex  domains $\Omega$ with Neumann boundary condition  is studied in \cite{LiMu, MR3210023, MR3335922}.  These works heavily rely on the following inequality
$$
\frac{\partial |\nabla v|^2}{\partial \nu}\le 0,\quad x\in\partial\Omega,
$$
where $\frac{\partial}{\partial\nu}$ denotes the outward normal derivative. Such an inequality is not available in the setting of metric graphs. For parabolic-parabolic models, i.e. $\tau>0$, we offer a new alternative approach which does apply to regular domains, see Section \ref{sec3}. We also note that the global stability of the positive constant solution 
for \eqref{parabolic-parabolic-eq} with $\tau=0$ on  regular domains $\Omega$ with Neumann boundary condition have been studied in \cite{IsSh1, MR3620027, MR2334836}. We adopt the approach established in \cite{MR3620027} to prove global stability of  the positive constant solution 
for \eqref{parabolic-parabolic-eq} with $\tau=0$.
}

{The rest of the paper is organized as follows. In Section \ref{sec2}, we study the local stability,  instability, and bifurcation  of the constant solution $(\frac{a}{b},\frac{a}{b})$ of
 \eqref{parabolic-parabolic-eq}, \eqref{NKsol}  and prove Theorems \ref{local-stability-thm} and \ref{prop4.4}.  In Section \ref{sec3}, we investigate the global stability of the constant solution $(\frac{a}{b},\frac{a}{b})$ of  \eqref{parabolic-parabolic-eq}, \eqref{NKsol} and prove Theorems 
\ref{global-stability-thm}  and \ref{4.4}.  We supply numerical computation of bifurcation points for several graphs in Section \ref{numerics}. Finally, in Appendix \ref{functionalspaces}, 
 we record several facts about fractional power spaces  generated by the Neumann--Kirchhoff Laplacian on compact metric graphs.

}

\section{{Local stability, instability,  and bifurcation of constant steady states}}\lb{sec2}

{ In this section, we study the local stability,  instability, and bifurcation  of the constant solution $(\frac{a}{b},\frac{a}{b})$ of
 \eqref{parabolic-parabolic-eq}, \eqref{NKsol}  via  spectral analysis of the linearizations of 
\eqref{parabolic-parabolic-eq}, \eqref{NKsol}
 about the steady state solution $(\frac{a}{b},\frac{a}{b})$.   We first recall a global well-posedness result
from \cite{HWS} in Section 2.1. We then prove Theorems \ref{local-stability-thm} and \ref{prop4.4} in Sections 2.2 and 2.3, respectively.}

\subsection{Well-posedness of Keller--Segel model on graphs}

  First, let us record a result concerning well-posedness of \eqref{parabolic-parabolic-eq} subject to vertex conditions {\eqref{NKsol}}  in $\elpee$ (see  \eqref{L-p-eq} in Appendix A for the definition of $L^p(\Gamma)$ and other functional spaces on metric graphs)  and regularity of solutions, in particular, their membership to the fractional power spaces $\cX^{\beta}_p$ generated by the Neumann--Kirchhoff Laplacian $\Delta$ and to the space of H\"older continuous functions $\hatt C^{\nu}(\overline{\Gamma})$ (see Appendix \ref{functionalspaces} for definition of $\cX_p^\beta$, $\hatt C^{\nu}(\overline{\Gamma})$).

\begin{theorem}{\cite{HWS}.} \lb{Lp} Let $\Gamma$ be a connected compact metric graph. Then there exists $p_0\geq 1$ such that the following assertions hold for $p\geq p_0$.  
	
	(1)  Assume that $\tau=0$.  Then for arbitrary $u_0\in \elpee$, \eqref{parabolic-parabolic-eq} has a unique global classical solution  $u=u(t, x; u_0)$, $v=v(t, x; u_0)$, $t\geq 0$ satisfying Neumann--Kirchhoff vertex conditions \eqref{NKsol}. For such a solution one has 
	\begin{align}\lb{rubp}
		u\in C((0, \infty), \hatt C^{\nu}(\overline{\Gamma}))\cap  C([0, \infty), \elpee)\cap C^{0,\beta}((0,  \infty), \cX^{\beta}_r),
	\end{align}	
	for arbitrary $r\geq1$, $\beta\in(0,1/8)$, $\nu<\beta$.  Moreover, if $u_0\in\hatt C(\overline{\Gamma})$ is non-negative {and not identically zero},  then $u=u(t, x; u_0)>0$, $v=v(t, x; u_0)>0$ for all { $t> 0$}, $x\in\Gamma$.
	
	(2) Assume that $\tau>0$ and let $(u_0, v_0)\in \elpee\times \wpee$. Then  \eqref{parabolic-parabolic-eq} has a unique global classical solution $u=u(x, t; u_0, v_0)$, $v=v(x, t; u_0, v_0)$, $t\in [0, \infty)$ satisfying Neumann--Kirchhoff vertex conditions \eqref{NKsol}. For such a solution one has
	\begin{align}
		\begin{split}\label{rmilduvsolsmax}
			&u\in C([0, \infty), L^p(\Gamma))\cap C((0, \infty), L^r(\Gamma))\cap C^{0,\beta}((0, \infty), \cX^{\beta}_r)\cap C^{0,\beta}((0, \infty), \hatt C^{\nu}(\overline{\Gamma})),  \\
			&v\in C([0, \infty), \wpee)\cap C^{0,\beta}((0, \infty), \hatt W^{2,r}(\Gamma))\cap C^{0,\beta}((0, \infty), \cX^{\beta}_r)\cap C^{0,\beta}((0, \infty), \hatt C^{\nu}(\overline{\Gamma})), 
		\end{split}
	\end{align}	
	for arbitrary $r\geq1$, $\beta\in(0,1/8)$, $\nu<\beta$. Moreover, if $u_0\in\hatt C(\overline{\Gamma})$, $v_0\in \hatt C^1(\overline{\Gamma})$ are non-negative {with $u_0\not\equiv 0$}  the $u=u(t, x; u_0, v_0)>0$, $v=v(t, x; u_0, v_0)> 0$ for all { $t> 0$}, $x\in\Gamma$.
\end{theorem}

We note that well-posedness of Keller--Segel model on subset of $\bbR^n$, $n\geq 1$ has been investigated by numerous authors, see, for example, \cite{MR4188348, MR3698165, MR3620027, MR2334836, MR2445771, MR2825180} and references therein. 

\subsection{Local asymptotic stability}
In this subsection we first  discuss  spectral properties of linearizations of parabolic-parabolic and parabolic-elliptic equations about the steady state solution $(a/b, a/b)$ and  then prove Theorem \ref{local-stability-thm}.  In particular, we show that the non-selfadjoint linearized operators have compact resolvents, hence, their spectra is discrete and compute (in general, complex) eigenvalues in terms of the eigenvalues of  Neumann-Kirchhoff Laplacian, see Lemmas \ref{eigenvalue-lm2} and \ref{eigenvalue-lm1}. Then we prove the following:
\begin{itemize}
\item if $\chi\in (0,\chi^*)$ then all eigenvalues of the linearized operators have negative real part,
\item if $\chi\in (\chi^*, \infty)$ then the linearized operators exhibit eigenvalues with positive real part. 
\end{itemize}
Let us introduce the following semi-linear mappings corresponding to parabolic-parabolic and parabolic-elliptic equations respectively
\begin{align}
	\begin{split}\label{nonlinopnew}
		&{\cF}(u,v,\tau, \chi):  \cD \times \cD {\times [0,\infty)}\times \mathbb{R} \rightarrow L^2(\Gamma),\ \\
		&{\cF}(u,v, \tau, \chi):=\begin{bmatrix}
			\partial_x\big(\partial_{x} u-\chi u\partial_xv\big)+u(a-bu)\\ \tau^{-1}(\partial_{xx}^2 v-v+u)
		\end{bmatrix},\\
	\end{split}
\end{align}
and 
\begin{align}
	\begin{split}\label{nonlinopnew0}
		&F(u,\chi):  \cD \times \mathbb{R} \rightarrow L^2(\Gamma)\times L^2(\Gamma),\ \\
		&F(u, \chi)u:=\partial_{xx}^2 u+\chi\partial_x\left(u \partial_x (\Delta-I)^{-1}u)\right)+u(a-bu), \\
\end{split}
\end{align}
where $\mathcal{D}$ is as in \eqref{D-eq}
and $\mathbb{R}^+=(0,\infty)$.
Let us recall, from \cite[Section 3.1.1, Theorem 1.4.19]{BK}, see also \cite{MR3748521}, that the spectrum of the Neumann--Kirchhoff Laplacian  $\Delta$ on a compact graph is discrete\footnote{in contrast to \cite{BK} we consider the positive Laplace operator $\Delta$ whose spectrum accumulates at $-\infty$} and bounded from above. 

\begin{lemma}\label{eigenvalue-lm2}
The linerization of $\cF(u,v,\tau, \chi)$ about $(a/b, a/b)$ is given by 
	\begin{equation}
	D_{(u,v)}\cF \left(a/b, a/b,\tau, \chi\right)=D(a,b,\tau, \chi),
	\end{equation}
where $D(a,b,\tau, \chi): \el{2}\times \el{2}\rightarrow \el{2}\times \el{2}$ is a non-selfadjoint block operator matrix given by 
\begin{align}
	&\dom\left(D(a,b,\tau, \chi)\right):= \dom(\Delta)\times\dom(\Delta),\\
	&D(a,b,\tau, \chi):=\begin{bmatrix}
		\Delta -aI_{\el{2}}& {-} \frac{\chi a}{b}\Delta\\
		\tau^{-1} I_{\el{2}}& \tau^{-1}{(}\Delta -I_{\el{2}}{)}
	\end{bmatrix}, 
\end{align}
where $\Delta$ denotes the Neumann--Kirchhoff Laplacian on a compact graph $\Gamma$, $a,b, \tau, \chi$  are positive constants. 

Then the spectrum of $D(a,b,\tau, \chi)$ is discrete, that is, it consist of isolated eigenvalues of finite multiplicity and it is given by 
\begin{equation}
\spec\left(D(a,b,\tau, \chi)\right)=\left\{\mu\in\bbC : \det (\lambda A+B-\mu T)=0, \lambda\in \spec(\Delta)\right\},
\end{equation}
where
\begin{align}
	A:=\begin{bmatrix}
		1 & 	-\frac{\chi a}{b}\\
		0& 1
	\end{bmatrix},
	B:=
	\begin{bmatrix}
		-a & 0\\
		1&-1
	\end{bmatrix}, 
T:=
	\begin{bmatrix}
		1& 0\\
		0&\tau 
	\end{bmatrix}.
\end{align}

Concretely, $\mu\in \spec\left(D(a,b,\tau, \chi)\right)$ if and only if 
\begin{equation}\label{muplus}
\mu=\frac{-\Big(1-(1+\tau)\lambda+a\tau\Big)+ \sqrt { \Big(1-(1+\tau)\lambda+a\tau\Big)^2-4 \tau \Big((a-\lambda)(1-\lambda)+\chi\frac{a}{b}\lambda\Big)}}{2\tau},
\end{equation}
or
\begin{equation}\label{muminus}
	\mu=\frac{-\Big(1-(1+\tau)\lambda+a\tau\Big)- \sqrt { \Big(1-(1+\tau)\lambda+a\tau\Big)^2-4 \tau \Big((a-\lambda)(1-\lambda)+\chi\frac{a}{b}\lambda\Big)}}{2\tau}.
\end{equation}
for some eigenvalue $\lambda$ of the Neumann--Kirchhoff Laplacian. 
\end{lemma}

\begin{proof}
In the first step we find the  eigenvalues of $D=D(a,b,\tau, \chi)$, in the second step we will prove that $D-\mu$ is boundedly invertible, that is,  $(D-\mu)^{-1}\in\cB(\el{2}\times \el{2})$   whenever $\mu\in\bbC$ is not an eigenvalue. 

{\it Step one.}  Let $\Delta_2:=\Delta\oplus \Delta$  and
\begin{align}
&\cA:=A\otimes   I_{\el{2}} ={\begin{bmatrix}
		 I_{\el{2}} & -\frac{\chi a}{b} I_{\el{2}}\\
		0&  I_{\el{2}}
	\end{bmatrix}},\\
&\cB:=A\otimes I_{\el{2}}=\begin{bmatrix}
		-aI_{\el{2}}& 0_{\el{2}}\\
		I_{\el{2}}& -I_{\el{2}}\\
	\end{bmatrix}, \lb{ab}\\
&\cT=T\otimes I_{\el{2}}=\begin{bmatrix}
	I_{\el{2}}& 0_{\el{2}}\\
	0_{\el{2}}& \tau I_{\el{2}}\\
\end{bmatrix}.
\end{align}
Then one has
\begin{align}
D=\cT^{-1}(\cA\Delta_2  +\cB),
\end{align} 
and $\mu$ is an eigenvalue of $D$ if and only if 
\begin{equation}\label{eq:kernelcond}
\ker \left({\cA\Delta _2 }+(\cB-\cT \mu) \right)\not=\{0\}. 
\end{equation}
Since $1\not\in\spec(\Delta)$, $0\not\in\spec(\cA)$ one has
\begin{equation}
\cA \Delta_2+(\cB-\cT \mu)=\cA(\Delta_2-I)(I+(\Delta_2-I)^{-1}(\cA^{-1}(\cB-\cT \mu)+I)), 
\end{equation}
and \eqref{eq:kernelcond} is equivalent to 
\begin{equation}\lb{kercond2}
\ker(I+(\Delta_2-I)^{-1}(\cA^{-1}(\cB-\cT \mu)+I))\not=\{0\}.
\end{equation}
Since $(\Delta_2-I)^{-1}$ is a Hilbert--Schmidt operator and $(\cA^{-1}(\cB-\cT \mu)+I)$ is bounded,  the operator
\begin{equation}
V_{\mu}:=(\Delta_2-I)^{-1}(\cA^{-1}(\cB-\cT \mu)+I)
\end{equation}
is trace class. Hence, $I+V_{\mu}$ is boundedly invertible if and only if $\det (I+V_{\mu})\not=0$, cf.,e.g, \cite[Theorem VII. 7.1]{MR1130394}. Next, we compute this perturbation determinant explicitly. Let $P_{t}:=\chi_{(t,\infty)}(\Delta_2)$ be the spectral projection of $\Delta_2$ corresponding to the interval $(t,\infty)$. Since the spectrum of $\Delta_2$ is discrete and bounded from above, one has $\dim\ran(P_{t})<\infty$, $t\in\bbR$ and $\lim\limits_{t\rightarrow -\infty} P_t=I_{\el{2}\times \el{2}}$. Then one has
\begin{align}
\begin{split}\lb{detcomp}
	\det(I+V_{\mu})&=\lim\limits_{t\rightarrow-\infty}\det(I_{\ran P_{t}}+P_{t}(\Delta_2-I)^{-1}(\cA^{-1}(\cB-\mu\cT)+I)P_{t})\\
	&=\lim\limits_{t\rightarrow-\infty}\det(I_{\ran P_{t}}+P_{t}(\Delta_2-I)^{-1}P_{t}(\cA^{-1}(\cB-\mu\cT )+I)P_{t})\\
	&=\lim\limits_{t\rightarrow-\infty}\prod_{\substack{\lambda\in\spec(\Delta_2)\\ \lambda>t} }\det (I_2+(\lambda-1)^{-1}(A^{-1}(B-\mu T)+I))\\
	&=\prod_{\substack{\lambda\in\spec(\Delta_2)} }\frac{\det (\lambda A +B-\mu T)}{\det (A) (\lambda-1)^{2}},
\end{split}
\end{align} 
where we used that fact that $(\cA^{-1}(\cB-\mu\cT )+I)$ and $P_t$ commute. The latter is inferred, for example, from the matrix representation of these operators with respect to spectral the decomposition 
\begin{equation}
\el{2}\oplus \el{2}=\bigoplus_{\lambda\in\spec(\Delta_2)} \ran\chi_{\{\lambda\}}(\Delta_2),
\end{equation}
\begin{align}
(\cA^{-1}(\cB-\mu\cT )+I)&=
\begin{bmatrix}
(A^{-1}(B-\mu T )+I_2)& 0_2&...\\
0_2& (A^{-1}(B-\mu T )+I_2)& ...\\
 \vdots & \vdots &\ddots\\
\end{bmatrix},\\
P_t&=
\begin{bmatrix}
	\lambda_1 I_2& 0_2&...& 0_2&0_2&...\\
	0_2& \lambda_2 I_2& ...& 0_2&0_2&...\\
	\vdots & \vdots &\ddots& \vdots&\vdots&\vdots\\
		0_2&0_2& ...&\lambda_k I_2&0_2&0_2\\
			0_2&0_2& ...&0_2&0_2&\ddots\\
\end{bmatrix}, 
\end{align} 
where $\lambda_1\geq  ... \geq \lambda_k$ are eigenvalues of $\Delta_2$ and $\lambda_k$ is the smallest eigenvalue satisfying $\lambda_k>t$.
Then \eqref{detcomp} yields \eqref{kercond2} which in turn shows that $\mu$ is an eigenvalue of $D$ { if} and only if $\det (\lambda A+B-\mu T)=0$ for some $\lambda\in \spec(\Delta)$, that is $\mu$ is as in \eqref{muplus}, \eqref{muminus}.

{\it Step two.} For $\mu\in \bbC$ one has
\begin{align}\lb{dminusmu}
D-\mu=\cT^{-1}\cA(\Delta_2-I)(I+(\Delta_2-I)^{-1}(\cA^{-1}(\cB-\cT \mu)+I)).
\end{align}
Let us pick $\mu\in\bbC$ such that $\det (\lambda A +B-\mu T)\not=0$ for all $\lambda\in\spec(\Delta_2)$\footnote{such a $\mu$ exists because $\spec(\Delta_2)$ is discrete}. Then by step one the operator { in}  the right-hand side of \eqref{dminusmu} is boundedly invertible, hence, 
\begin{align}
(D-\mu)^{-1}=(I+(\Delta_2-I)^{-1}(\cA^{-1}(\cB-\cT \mu)+I))^{-1}(\Delta_2-I)^{-1}\cA^{-1}\cT.
\end{align}
Since $(\Delta_2-I)^{-1}$ is compact and all other factors are bounded we infer that $(D-\mu)^{-1}$ is also compact. Therefore, the spectrum of $D$ purely discrete and consists of eigenvalues given by \eqref{muplus}, \eqref{muminus}. 
\end{proof}

\begin{lemma}
\label{eigenvalue-lm1}
The linerization of $F(u, \chi)$ about $a/b$ is given by 
\begin{equation}
	D_{u}F \left(a/b, \chi\right)=D(a,b,\chi),
\end{equation}
where $D(a,b, \chi):  \el{2}\rightarrow  \el{2}$ is a self-adjoint operator given by
\begin{align}
	&\dom\big(D(a,b, \chi)\big):= \dom(\Delta),\\
	&D(a,b,\chi):=\Delta { - }\chi\frac{a}{b} \left(\Delta-I\right)^{-1} -\left(a-\chi\frac{a}{b}\right),\ 
\end{align}
where $\Delta$ denotes the Neumann--Kirchhoff Laplacian on a compact graph $\Gamma$, $a,b, \chi$  are positive constants. Then the spectrum of $D(a,b, \chi)$ is discrete, that is, it consists of isolated eigenvalues and $\mu\in\spec\big(D(a,b, \chi)\big)$ if and only if
\begin{equation}
\mu=\lambda-\frac{\chi a}{b(1-\lambda)}-\left(a-\chi\frac{a}{b}\right)
\end{equation}
for some $\lambda\in \spec(\Delta)$.
\end{lemma}
\begin{proof}
Let $f(t):=t{ -} \chi a b^{-1}(1-t)^{-1}-(a-\chi ab^{-1})$, $t\leq0$. Then $D(a,b,\chi)=f(\Delta)$ and the assertions follow form the spectral theorem combined with the fact that $\Delta$ has compact resolvent.
\end{proof}

We now ready to prove Theorem \ref{local-stability-thm}.

\begin{proof}[Proof of Theorem \ref{local-stability-thm}]  { To prove local asymptotic stability of $(\frac{a}{b},\frac{a}{b})$ in Part (1) and the instability of $(\frac{a}{b},\frac{a}{b})$  in Part (2), it suffices to show that the spectrum of the linearized operator is a subset of $\{z\in\bbC: \Re z <0\}$ if $0<\chi<\chi^*$, and that it intersects  the set $\{z\in \bbC:\Re z>0\}$  if $\chi>\chi^*$, cf., e.g.,  \cite[Theorem 5.1.1]{Henry}. We prove this for the cases $\tau=0$ and $\tau>0$ separately.}

\smallskip

{First, consider the case that $\tau=0$.}   Recall  \eqref{nonlinopnew0} and its linearization $D(a,b, \chi)$ from  Lemma \ref{eigenvalue-lm1}. Then $\mu\in\spec\big(D(a,b, \chi)\big)$ if and only if
\begin{equation}
	\mu=\lambda{ -a}+\chi\frac{a}{b}\left(1-\frac{1}{1-\lambda}\right).
\end{equation}
for some $\lambda\in\spec(\Delta)$. Since $\lambda\leq 0$, one has $\mu=\mu(\chi)$ is a non-decreasing and vanishes at  
\begin{equation}
\chi(\lambda):=\frac{ b(\lambda-a)(1-\lambda)}{a\lambda}, \quad \lambda\in\spec(\Delta).
\end{equation}
Hence, $\mu<0$ whenever $0<\chi<\min\{\chi(\lambda): \lambda\in\spec(\Delta)\}=\chi^*$. That is 
\begin{equation}
\spec\left(D(a,b, \chi)\right)\subset \{z\in\bbC: \Re z <0\}, \chi\in(0,\chi^*).
\end{equation}
{Moreover, if $\chi>\chi^*$, then for some $\lambda\in \spec(\Delta)$ one has $\mu=\lambda{-a}+\chi\frac{a}{b}\left(1-\frac{1}{1-\lambda}\right)>0$, which concludes the proof of the case $\tau=0$.}

\smallskip

{ Next, we consider the case$\tau>0$.}  Recall  \eqref{nonlinopnew0} and its linearization $D(a,b, \tau, \chi)$ from Lemma \ref{eigenvalue-lm1}. Then the spectrum of $D(a,b, \tau, \chi)$ is given by the following eigenvalues 
\begin{equation}
\mu_{\pm}(\lambda,\chi)=\frac{-\Big(1-(1+\tau)\lambda+a\tau\Big)\pm\sqrt { \Big(1-(1+\tau)\lambda+a\tau\Big)^2-4 \tau \Big((a-\lambda)(1-\lambda)+\chi\frac{a}{b}\lambda\Big)}}{2\tau}
\end{equation}
for
$\lambda\in\spec(\Delta)$. First, let us observe that 
\begin{equation}
1-(1+\tau)\lambda+a\tau>0. 
\end{equation}
Hence $\Re(\mu_-(\lambda,\chi))<0$ for all $\lambda\in\spec(\Delta)$, $\chi>0$. 

To show $\Re(\mu_+(\lambda,\chi))<0$ for all  $\lambda\in\spec(\Delta)$, $\chi\in(0, \chi^*)$, we first note that 
\begin{equation}
\mu_+(0,\chi)=\frac{-(1+a\tau)+\sqrt { \Big(1+a\tau\Big)^2-4 \tau a}}{2\tau}<0.
\end{equation} 
If $\lambda\in\spec(\Delta)\setminus\{0\}$ then either 
\begin{equation}
 \Big(1-(1+\tau)\lambda+a\tau\Big)^2\geq 4 \tau \Big((a-\lambda)(1-\lambda)+\chi\frac{a}{b}\lambda\Big),
\end{equation}
in which case $\Re(\mu_+(\lambda,\chi))<0$, or
\begin{equation}
	\Big(1-(1+\tau)\lambda+a\tau\Big)^2< 4 \tau \Big((a-\lambda)(1-\lambda)+\chi\frac{a}{b}\lambda\Big),
\end{equation}
in which case $\chi\mapsto \mu_+(\lambda,\chi)$ is a real-valued, non-decreasing function of $\chi$. In the latter case, the equation $\mu_+(\lambda,\chi)=0$ reads
\begin{equation}
\Big(1-(1+\tau)\lambda+a\tau\Big)=\sqrt { \Big(1-(1+\tau)\lambda+a\tau\Big)^2-4 \tau \Big((a-\lambda)(1-\lambda)+\chi\frac{a}{b}\lambda\Big)}
\end{equation}
and yields
\begin{equation}
\chi=\frac{ b(\lambda-a)(1-\lambda)}{a\lambda}, \lambda\in\spec(\Delta).
\end{equation}
Hence, $\Re(\mu_+(\lambda,\chi))=\mu_+(\lambda,\chi)<0$ whenever $\chi\in(0,\chi^*)$ as required.  { To finish the proof, we note that if $\chi>\chi^*$ then there exists $\lambda\in \spec(\Delta)$ such that $\mu_+(\lambda,\chi) >0$. }
\end{proof}

\subsection{Local bifurcation}

{ In this subsection, we prove Theorem  \ref{prop4.4} via  Crandall--Rabinowitz's Theorem, cf. \cite[Theorem 8.3.1]{MR1956130}.}

\begin{proof}[Proof of Theorem \ref{prop4.4}]Our goal is to verify conditions of Crandall--Rabinowitz's Theorem as stated in \cite[Theorem 8.3.1]{MR1956130}. To that end, we first note that steady state  solutions of both parabolic-parabolic and parabolic-elliptic equations stem for the same system
	\begin{align}
		\begin{cases}
				\partial_x\big(\partial_{x} u-\chi u\partial_xv\big)+u(a-bu)=0,\\ 
				\partial_{xx}^2 v-v+u=0,
		\end{cases}
	\end{align}
or, equivalently, $\mathcal{F}(u,v,1, \chi)=0$, cf. \eqref{nonlinopnew}. The partial derivative with respect to $(u,v)$ of this nonlinear mapping is given by   
	\begin{align}
		&L:=D_{(u,v)}\cF \left(u, v, 1, \chi\right)\in\cB(\cD\times \cD\times\bbR, L^2(\Gamma)\times L^2(\Gamma)),	\\
		&D_{(u,v)}\cF(u,v, \chi)=L_2+L_1\text{ where},\\
		&L_2\begin{bmatrix}
			f\\g
		\end{bmatrix} := 
		\begin{bmatrix}
			I_{L^2(\Gamma)}& -\chi u I_{L^2(\Gamma)}\\
			0_{L^2(\Gamma)}&I_{L^2(\Gamma)}
		\end{bmatrix}
		\begin{bmatrix}
			\Delta & 	0_{L^2(\Gamma)}\\
			0_{L^2(\Gamma)}&\Delta
		\end{bmatrix},\lb{el2}\\
		&L_1\begin{bmatrix}
			f\\g
		\end{bmatrix} := \begin{bmatrix}
			-\chi v'f'-\chi v''f-\chi u'g'+af-2buf \\ -g +f
		\end{bmatrix},\lb{el1}
	\end{align}
	here $f,g\in\cD$ and $\cD$ is considered as a Banach space with $\hatt W^{2,2}(\Gamma)-$norm.  This shows that $F\in C^2\left( \cD\times \cD\times\bbR, L^2(\Gamma)\times L^2(\Gamma)\right)$. Next, we show that $D_{(u,v)}\cF(u,v,1, \chi)$ is Fredholm with index zero as an operator from  $\cD\times \cD\times\bbR$  to $L^2(\Gamma)\times L^2(\Gamma)$.  Let us recall that the Neumann--Kirchhoof Laplacian $\Delta\in \cB(\hatt W^{2,2}(\Gamma), L^2(\Gamma))$ is Fredholm with index zero and the first term in the right-hand side of \eqref{el2}  is Fredhlom in $L^2(\Gamma)\times L^2(\Gamma)$ with index zero. Therefore by \cite[Theorem 3.16]{MR929030}, $L_2$ is Fredholm with index zero as a mapping from $\hatt W^{2,2}(\Gamma)$ to $L^2(\Gamma)$. Next, $L_1\in\cB(\hatt W^{2,2}(\Gamma), \hatt W^{1,2}(\Gamma))$ and the embedding $\hatt W^{1,2}(\Gamma)\hookrightarrow L^2(\Gamma)$ is compact, therefore $L_1$ is compact as a mapping from $\hatt W^{2,2}(\Gamma)$ to $L^2(\Gamma)$. Thus by \cite[Theorem 3.17]{MR929030} the operator $L=L_1+L_2$ is a Fredholm with index zero. 
	
	Recalling $\cA, \cB, \Delta_2$ from \eqref{ab} we obtain
	\begin{align}
		L&=D_{(u,v)}\cF \left(a/b, a/b, 1, \chi\right)=\cA \Delta_2+\cB.
	\end{align}
	Hence, one has
	\begin{align}
		\ker L=\ker \left( \begin{bmatrix}
			\Delta & 	0_{L^2(\Gamma)}\\
			0_{L^2(\Gamma)}&\Delta
		\end{bmatrix}+\begin{bmatrix}
			\left(-a{+\frac{\chi a}{b}}\right)I_{L^2(\Gamma)}& 	{ -\frac{\chi a}{b}I_{L^2(\Gamma)}}\\
			I_{L^2(\Gamma)}&-I_{L^2(\Gamma)}
		\end{bmatrix} \right).
	\end{align}
	Then $\ker L\not=\{0\}$ if and only if for some $\lambda\in\spec(\Delta)$ one has
	\begin{equation}
		\det \left( \begin{bmatrix}
			\lambda& 	0_{L^2(\Gamma)}\\
			0_{L^2(\Gamma)}&\lambda
		\end{bmatrix}+\begin{bmatrix}
			\left(-a+{\frac{\chi a}{b}}\right)I_{L^2(\Gamma)}& 	{-\frac{\chi a}{b}I_{L^2(\Gamma)}}\\
			I_{L^2(\Gamma)}&-I_{L^2(\Gamma)}
		\end{bmatrix} \right)=0,
	\end{equation}
	that is, if and only if one has
	\begin{equation}
		\left(a-\frac{\chi a}{b}-\lambda\right)(1-\lambda)+\frac{\chi a}{b}=0,
	\end{equation}
	or, equivalently, the identity \eqref{chieq} holds. By assumptions, we then obtain $\dim \ker (L)=1$ and $L\xi_0=0,\ \xi_0:=[
	u, v]^{\top}\varphi$. 
	
	Let us now show that the transversality condition in \cite[Theorem 8.3.1]{MR1956130} is also satisfied. That is, for  $K:=D^2_{(u,v),\chi}\cF(a/b, a/b,1, \chi )$ we show that $K [\xi_0, 1]^{\top}\not \in \ran (L)$. It suffices to show $\ker(L^*)=0$. Since $L^*=(\cA^{-1})^{*}(\Delta_2+(\cB\cA^{-1})^*)$, we note that $\ker(L^{*})\not =\{0\}$ yields a $\lambda\in\spec(\Delta)$ such that 
	\begin{equation}
		\det( \lambda+(BA^{-1})^*)=0,
	\end{equation}
	that is 
	\begin{equation}
		(\lambda-a)(\lambda-1)=0,
	\end{equation}
	which contradicts $\lambda\leq 0$, $a>0$. 
\end{proof}

\section{Global stability of constant steady states}\lb{sec3}

{ In this section, we study 
the global stability of the constant solution $(\frac{a}{b},\frac{a}{b})$   of \eqref{parabolic-parabolic-eq}, \eqref{NKsol} and prove Theorems \ref{global-stability-thm} and \ref{4.4}. 
Throughout this section, $C$ denotes a positive constant independent of $a,b, \chi$ and the solutions  of \eqref{parabolic-parabolic-eq}, \eqref{NKsol}. We first establish some lemmas and then prove  Theorems \ref{global-stability-thm}  and \ref{4.4}. }

\begin{lemma}
\label{bounds-lm1}
Assume the setting of Theorem \ref{global-stability-thm}. Then for $T>0$ one has
\begin{equation}
\label{bounds-eq1}
\limsup_{t\to\infty}  \int_\Gamma  u(t;u_0,v_0)dx\le \frac{a|\Gamma|}{b}.
\end{equation}
\end{lemma}

\begin{proof}
Integrating both sides  of the first equation in \eqref{parabolic-parabolic-eq} over $\Gamma$ we obtain
$$
\frac{d}{dt}\int_\Gamma u(t;u_0,v_0)dx=a\int_\Gamma u(t;u_0,v_0)dx-b\int_\Gamma u^2(t;u_0,v_0)dx.
$$
Combining this with
$$
\int_\Gamma u^2(t;u_0,v_0)dx\ge \frac{1}{|\Gamma|}\left(\int_\Gamma u(t;u_0,v_0)dx\right)^2,
$$
we arrive at
\begin{equation}
\label{L1-estimate-eq1}
\frac{d}{dt}\int_\Gamma u(t;u_0,v_0)dx\le a\int_\Gamma u(t;u_0,v_0)-\frac{b}{|\Gamma|}\left(\int_\Gamma u(t;u_0,v_0)\right)^2dx.
\end{equation}
Then, since $f(t)=|\Gamma|ab^{-1}$ solves the differential equation $f'=af-b(|\Gamma|)^{-1}f^2$, the comparison principle yields
$$
\limsup_{t\to\infty}\int_\Gamma u(t;u_0,v_0)dx\le \frac{a|\Gamma|}{b},
$$
as asserted.
\end{proof}

\begin{lemma}
\label{bounds-lm2}
Assume the setting of Theorem \ref{global-stability-thm}. Then there exists a constant $C=C(\Gamma)>0$ such that 
\begin{equation}
\label{bounds-eq2}
\limsup_{t\to\infty} \| v(t)\|_{ \hatt C^1(\overline\Gamma)} {\le C \frac{a}{b}}.
\end{equation}
\end{lemma}

\begin{proof}
Let us fix $\beta\in \left(\frac{1}{2},1\right)$ and $q>1$ satisfying $
2\beta -q^{-1}>1$. The by Theorem \ref{analytic-semigroup-thm2} one has
\begin{equation}\lb{emb1}
\cX_q^\beta \hookrightarrow \hatt C^1(\overline\Gamma).
\end{equation}
By the Duhamel principle, the second equation in \eqref{parabolic-parabolic-eq} yields
\[v(t)=e^{(\Delta-I)\frac{(t-t_0)}{\tau}}v(t_0)+\frac{1}{\tau}\int_{t_0}^{t}e^{(\Delta-I)\frac{t-s}{\tau}}u(s)ds,\]
where we abbreviated $v(t)=v(t;u_0,v_0)$. Then  one obtains
\begin{align}
 & { \| v(t)\|_{\hatt C^1(\overline \Gamma)}} \underset{\eqref{emb1}}{\leq}    C \|v(t)\|_{\cX_q^\beta}\le C \|  e^{(\Delta-I)\frac{t-t_0}{\tau}}(I-\Delta)^\beta v(t_0)\|_{\elq} \\
 &\hspace{5cm}+C\int_{t_0}^{t}\|(I-\Delta)^\beta  e^{(\Delta-I)\frac{t-s}{2\tau}}e^{(\Delta-I)\frac{t-s}{2\tau}}u(s)\|_{\elq}ds \nonumber\\
&\le C \|  e^{(\Delta-I)\frac{t-t_0}{\tau}}(I-\Delta)^\beta v(t_0)\|_{\elq}\\
&\hspace{4cm}+C \int_{t_0}^{t}  \Big(\frac{t-s}{2\tau}\Big)^{-\beta}e^{-\frac{t-s}{4\tau}} \| e^{(\Delta-I)\frac{t-s}{2\tau}}u(s)\|_{\elq} ds\nonumber\\
&\le C \|  e^{(\Delta-I)\frac{t-t_0}{\tau}}(I-\Delta)^\beta v(t_0)\|_{\elq}\\
&\hspace{4cm}+C\int_{t_0}^{t}\left(\frac{t-s}{2\tau}\right)^{-\beta-\frac{1}{2}(1-\frac{1}{q})}e^{-\frac{t-s}{2\tau}}\|u(s)\|_{\el{1}}ds\\
 &\leq  Ce^{\frac{t_0-t}{2\tau}}\| (I-\Delta)^\beta v(t_0)\|_{L^q(\Gamma)} +C\sup_{t\in [t_0,\infty)} \|u(t)\|_{\el{1}} \int_{t_0}^{t}\left(\frac{t-s}{2\tau}\right)^{-\beta-\frac{1}{2}(1-\frac{1}{q})}e^{-\frac{t-s}{2\tau}}ds.
\end{align}
{This implies  \eqref{bounds-eq2} by choosing sufficiently large $t_0$ such that $\sup_{t\in[t_0,\infty)}\|u(t)\|_{L^1(\Gamma)}\le 2|\Gamma|\frac{a}{b}$ and then letting $t\to\infty$. }
\end{proof}

\begin{lemma}
\label{bounds-lm3}
Assume the setting of Theorem \ref{global-stability-thm}. Then one has
\begin{equation}
\label{bounds-eq3}
\limsup_{t \to \infty} \|u(t)\|_{\el{4}}\leq  C\left({\frac{a}{b}}+\frac{a^2\chi^2}{b^3}\right).
\end{equation}
\end{lemma}

\begin{proof} Multiplying both sides of the first equation in \eqref{parabolic-parabolic-eq} by $u^3$ and integrating over $\Gamma$ we obtain
	\begin{align}\lb{firstequmultiplied}
	\int_{\Gamma}u_tu^3dx=\int_{\Gamma}u^3\partial_x\big(\partial_{x} u-\chi u\partial_xv\big)dx+\int_{\Gamma}u^4(a-bu)dx.
	\end{align}
	We note that
\begin{align}\lb{ibp1}
	\int_{\Gamma}  u^3u_{xx}dx&= -\int_{\Gamma} (u^3)_xu_{x}dx+\sum_{\theta\in \cV}\sum_{e\sim \theta}u_e^3(\theta)\partial_{\nu}u_e(\theta)\\
	& = -\int_{\Gamma} (u^3)_xu_{x}dx+\sum_{\theta\in \cV}u^{3}(\theta)\sum_{e\sim \theta}\partial_{\nu}u_e(\theta)=-\int_{\Gamma} (u^3)_xu_{x}dx=-3\int_{\Gamma} (uu_{x})^2dx,
\end{align}
where we used the fact that $u$ satisfies the Neumann-Kirchhoff vertex conditions. Similarly, one has 
\begin{align}
	\begin{split}
\lb{ibp2}
&\int_{\Gamma}  u^3\partial_x(uv_{x})dx= -\int_{\Gamma}(u^3)_xuv_{x}dx+\sum_{\theta\in \cV}\sum_{e\sim \theta}u^4_e(\theta)\partial_{\nu}v_e(\theta)=-\int_{\Gamma}(u^3)_xuv_{x}dx.\\
	\end{split}
\end{align}
Therefore, we have
\begin{equation}\lb{firsteqm2}
	\frac{1}{4} \frac{d}{dt}\int_{\Gamma} u^{4}dx=  -3 \int_{\Gamma}  (uu_{x})^2dx +3\chi  \int_{\Gamma}u^{3}u_x v_x dx +  \int_{\Gamma}u^4(a-bu)dx.
\end{equation}
We note that Young's inequality with exponents $2,2$ yields 
\begin{align}
\begin{split}
\chi  \int_{\Gamma}u^{3}u_x v_x dx&= \int_{\Gamma}(uu_x)(\chi u^2v_{x})dx\leq  \frac{\chi^2}{4}\int_{\Gamma}u^{4} |v_x|^2dx+\int_{\Gamma} (uu_x)^2dx\lb{young1}.
\end{split}
\end{align}
By H\"older's inequality gives
\begin{equation}\lb{holder1}
\int_\Gamma u^5 dx\geq \frac{1}{|\Gamma|^{1/5}}\left(\int_{\Gamma} u^4dx\right)^{5/4}. 
\end{equation}
Combining these inequalities with \eqref{firsteqm2} we obtain
  \begin{align*}
    \frac{1}{4}\frac{d}{dt}\int_{\Gamma}u^{4}dx&=-3 \int_{\Gamma}  (uu_{x})^2dx +3\chi  \int_{\Gamma}u^{3}u_x v_x dx +  \int_{\Gamma}u^4(a-bu)dx\\
&\underset{\eqref{young1}}{\leq }       \frac{3\chi^2}{4}\int_{\Gamma}u^{4} |v_x|^2dx  +a \int_{\Gamma}u^{4}dx
-b\int_\Gamma  u^{5}dx\\
&{ \le    \frac{3\chi^2}{4}\|v\|_{\hatt C^1(\overline \Gamma)}^2 \int_{\Gamma}u^{4} dx  +a \int_{\Gamma}u^{4}dx
-\frac{b}{|\Gamma|^{\frac{1}{5}}}\Big(\int_\Gamma  u^{4}dx\Big)^{\frac{5}{4}}}\\
&{=\Big ( \frac{3\chi^2}{4}\|v\|_{\hatt C^1(\overline \Gamma)}^2  +a  -  \frac{b}{|\Gamma|^{\frac{1}{5}}} \big(\int_{\Gamma}u^{4}dx\big)^{\frac{1}{4}}\Big) \int_\Gamma  u^{4}dx.}
\end{align*}
{This together with \eqref{bounds-eq2} imply \eqref{bounds-eq3}. }
\end{proof}

\begin{lemma}
\label{bounds-lm4}
Assume the setting of Theorem \ref{global-stability-thm}. Then for arbitrary $\gamma\in(0,1/2)$ one has
\begin{equation}
\label{bounds-eq4}
\limsup_{t \to \infty}\| (I-\Delta)^{\gamma }u(t)\|_{\el{2}}\leq C{ \left(\frac{2a^2}{b^2}+\frac{a^2\chi}{b^2}+\frac{a^2\chi^2}{b^3}+\frac{a^3\chi^3}{b^4}+\frac{a^4\chi^4}{b^5} \right).}
\end{equation}
\end{lemma}

\begin{proof}
By Duhamel's principle, the first equation in \eqref{parabolic-parabolic-eq} yields
\begin{align}
\begin{split}\lb{lem4.4e1}
u(t)=e^{(\Delta-I)(t-t_0)}u(t_0)&-\chi\int_{t_0}^{t}e^{(\Delta-I)(t-s)}\partial_x\left(u(s)\partial_x v(s)\right)ds\\
&+\int_{t_0}^{t}e^{(\Delta-I)(t-s)}u(s)(a-bu(s))ds.
\end{split}
\end{align}
Let us note the following auxiliary inequalities 
\begin{align}
	\begin{split}\lb{lem4.4e2}
	&\|(I-\Delta)^{\gamma} e^{(\Delta-I)\frac{t-s}{2}}  e^{(\Delta-I)\frac{t-s}{2}} \partial_x(u(s)\partial_x v(s))\|_{\el{2}}\\
	&\underset{\eqref{fracpow}}{\leq} C (t-s)^{-\gamma}e^{\frac{t-s}{4}}\|e^{(\Delta-I)\frac{t-s}{2}} \partial_x(u(s)\partial_x v(s))\|_{\el{2}}\\
	&\underset{\eqref{tgradlplq}}{\leq} C (t-s)^{-\gamma-1/2}e^{\frac{t-s}{4}}\|u(s)\partial_x v(s)\|_{\el{2}}\\
	&\leq C (t-s)^{-\gamma-1/2}e^{\frac{t-s}{2}}{ \sup\limits_{r\geq t_0}\|\partial_x v(r)\|_{\el{\infty}}}\|u(s)\|_{\el{2}},
	\end{split}
\end{align}
\begin{align}
	\begin{split}\lb{lem4.4e3}
		&\|(I-\Delta)^{\gamma } e^{(\Delta-I)(t-s)}u(s)(a-bu(s))\|_{\el{2}}\\
		&\underset{\eqref{fracpow}}{\leq} C (t-s)^{-\gamma}e^{\frac{t-s}{2}}\|au(s)-bu^2(s)\|_{\el{2}}\\
		&\leq C (t-s)^{-\gamma}e^{\frac{t-s}{2}}\left(a\|u(s)\|_{\el{2}}+b{\|u(s)\|_{\el{4}}^{2}}\right),
	\end{split}
\end{align}
and
\begin{equation}
\label{lem4.4e00}
{ \|u(s)\|_{L^2(\Gamma)}\le |\Gamma|^{\frac{1}{4}}\|u(s)\|_{L^4(\Gamma)}.}
\end{equation}
Combining \eqref{lem4.4e1}, \eqref{lem4.4e2}, \eqref{lem4.4e3}, { \eqref{lem4.4e00}} we obtain
\begin{align*}
\| (I-\Delta)^{\gamma}u(t)\|_{\el{2}}&\le Ce^{\frac{t-t_0}2}\|(I-\Delta)^{\gamma}u(t_0)\|_{\el{2}}\\
&\qquad { +} C \chi \int_{t_0}^{t}(t-s)^{-\gamma-1/2}e^{\frac{t-s}{2}}\sup\limits_{ r\geq t_0}\|\partial_x v(r)\|_{\el{\infty}}\|u(s)\|_{\el{2}}ds\\
& \qquad +C\int_{t_0}^{t} (t-s)^{-\gamma}e^{\frac{t-s}{2}}\left(a\|u(s)\|_{\el{2}}+b\|u(s)\|_{\el{4}}^{{ 2}}\right)ds\\
&{\le    Ce^{\frac{t-t_0}2}\|(I-\Delta)^{\gamma}u(t_0)\|_{\el{2}}}\\
&\qquad { +  C \chi  \sup\limits_{r\geq t_0}\|\partial_x v(r)\|_{\el{\infty}}
\sup\limits_{t\ge t_0}\|u(r)\|_{\el{4}} \int_{t_0}^{t}(t-s)^{-\gamma-1/2}e^{\frac{t-s}{2}}ds}\\
& \qquad{+C \left(a\sup\limits_{r\ge t_0}\|u(r)\|_{\el{4}}+b\sup\limits_{r\ge t_0}\|u(r)\|_{\el{4}}^{ 2}\right) \int_{t_0}^{t} (t-s)^{-\gamma}e^{\frac{t-s}{2}}ds}
\end{align*}
 {This together with \eqref{bounds-eq2}  and \eqref{bounds-eq3} implies   \eqref{bounds-eq4} (choosing $t_0$ sufficiently large).}
\end{proof}

\begin{lemma}
\label{bounds-lm5}
Assume the setting of Theorem \ref{global-stability-thm}. Then one has
\begin{equation}
\label{bounds-eq6}
\limsup_{t \to \infty} \|\Delta v(t)\|_{\el{\infty}}\leq \kappa(a,b,\chi),
\end{equation}
where $\kappa(a,b,\chi)$ is as in \eqref{kap}.
\end{lemma}
\begin{proof}
The second equation in \eqref{parabolic-parabolic-eq} together with Duhamel's principle yields
\begin{equation}
v(t)=e^{(\Delta-I)\frac{t-t_0}{\tau}}v(t_0)+\frac{1}{\tau}\int_{t_0}^{t}e^{(\Delta-I)\frac{t-s}{\tau}}u(s)ds,
\end{equation}
hence,
\begin{equation}\lb{duhamel}
	(I-\Delta)v(t)=e^{(\Delta-I)\frac{t-t_0}{\tau}}(I-\Delta)v(t_0)+\frac{1}{\tau}\int_{t_0}^{t}(I-\Delta)e^{(\Delta-I)\frac{t-s}{\tau}}u(s)ds.
\end{equation}
Let us estimate $\el{\infty}$ norm of each term above. To that end, we first note the following embedding
\begin{equation}\lb{emb2-0}
{
\cX^{\alpha }_2\hookrightarrow \hatt C^{\nu }(\overline{\Gamma}), \,\, \frac{1}{4}<\alpha<1,\,\,   0<\nu  <2\alpha-\frac{1}{2}}.
\end{equation}
{ Choose $\frac{1}{4}<\alpha<\frac{1}{2}<\gamma<1$.}
One has
\begin{align}
	&\int_{t_0}^ t \|(I-\Delta) e^{(\Delta-I)\frac{t-s}{\tau}}u(s)\|_{\el{\infty}} ds\\
&\hspace{1.9cm}\le  \int_{t_0}^ t \|(I-\Delta) e^{(\Delta-I)\frac{t-s}{\tau}}u(s)\|_{C^\nu(\overline{\Gamma})} ds\\
&\hspace{1.9cm}\le  C \int_{t_0}^ t \|(I-\Delta) e^{(\Delta-I)\frac{t-s}{\tau}}u(s)\|_{ \cX^{\alpha }_2 } ds\\
&\hspace{1.9cm}\le  C \int_{t_0}^ t \|(I-\Delta)^{1+\alpha} e^{(\Delta-I)\frac{t-s}{\tau}}u(s)\|_{ L^2(\Gamma) } ds\\
&\hspace{1.9cm}\le  C \int_{t_0}^ t \|(I-\Delta)^{1+\alpha-\gamma} e^{(\Delta-I)\frac{t-s}{\tau}}(I-\Delta)^{\gamma}u(s)\|_{ L^2(\Gamma) } ds\\
&\hspace{1.9cm}\le  C \int_{t_0}^ t (t-s)^{-(1+\alpha-\gamma)} e^{-\frac{t-s}{2\tau}}\|(I-\Delta)^{\gamma}u(s)\|_{ L^2(\Gamma) } ds
\end{align}
Combining this with {
\begin{align*}
 \| (I-\Delta)e^{(\Delta-I) \frac{t-t_0}{\tau}}v(t_0)\|_{\el{\infty}} &\le \| (I-\Delta)e^{(\Delta-I) \frac{t-t_0}{\tau}}v(t_0)\|_{C^\nu(\overline{\Gamma})} 
\\
&\le \| (I-\Delta)e^{(\Delta-I) \frac{t-t_0}{\tau}}v(t_0)\|_{\cX^{\alpha }_2}\\ 
&\le C e^{\frac{t_0-t}{2\tau}}\| (I-\Delta)^{1+\alpha } v(t_0)\|_{L^2(\Gamma)}\\
\end{align*}
}
we obtain
{
\begin{align}
	\|(I-\Delta) v(t)\|_{\el{\infty}} &\le C e^{\frac{t_0-t}{2\tau}}\| (I-\Delta)^{1+\alpha } v(t_0)\|_{L^2(\Gamma)}\\
	&\qquad+  C \int_{t_0}^ t \|(t-s)^{-(1+\alpha-\gamma)} e^{-\frac{t-s}{2\tau}}\|(I-\Delta)^{\gamma}u(s)\|_{ L^2(\Gamma) } ds.
\end{align}}
This inequality together with \eqref{bounds-eq2} and \eqref{bounds-eq4} (which is applicable since { $\frac{1}{4}<\alpha<\frac{1}{2}<\gamma<1$}) yield \eqref{bounds-eq6}.
\end{proof}

We now prove Theorem \ref{global-stability-thm}.

\begin{proof}[Proof of Theorem \ref{global-stability-thm}]
	Let us observe that for arbitrary $\varepsilon>0$ Lemma \ref{bounds-lm5} yields a $\chi-$independent constant such that for arbitrary $x\in\overline{\Gamma}$, $t>0$ and $\chi>0$ one has
	\begin{equation}\lb{deltavbound}
	-(\kappa(a,b,\chi)+\varepsilon)\leq v_{xx}(t,x)\leq \kappa(a,b,\chi)+\varepsilon.
	\end{equation}
	Employing \eqref{deltavbound} we obtain the following inequalities
	\begin{align*}
		u_t&=u_{xx}-\chi(u_x v)_x+u(a-bu)\\
		&=u_{xx}-\chi u_x v_x-\chi u_{xx} v+u(a-bu)\\
		&\ge u_{xx}-\chi u_x v_x- (\kappa(a,b,\chi)+\varepsilon) u+ u(a-bu),
	\end{align*}
	and
	\begin{align*}
		u_t&=u_{xx}-\chi(u_x v)_x+u(a-bu)\\
	&=u_{xx}-\chi u_x v_x-\chi u_{xx} v+u(a-bu)\\
		&\le  u_{xx}-\chi  u_x v_x+(\kappa(a,b,\chi)+\varepsilon) u+ u(a-bu).
	\end{align*}
	Therefore the partial differential equation 
	\begin{equation}
	\varphi_t=\varphi_{xx}-\chi \varphi_x v_x + (\kappa(a,b,\chi)+\varepsilon)  \varphi+ \varphi(a-b\varphi), 
	\end{equation}
exhibits the subsolution $u$ {and} a constant solution
\begin{align}
&\overline{u}_\varepsilon :=\frac{a+\chi(\kappa(a,b,\chi)+\varepsilon) }{b},
\end{align}
and, similarly $u$ is a supersolution of 
\begin{equation}
	\varphi_t=\varphi_{xx}-\chi \varphi_x v_x -  \chi(\kappa(a,b,\chi)+\varepsilon)  \varphi+ \varphi(a-b\varphi), 
\end{equation}
while a constant solution is given by
\begin{align}
	&\underline{u}_\varepsilon :=\frac{a-\chi(\kappa(a,b,\chi)+\varepsilon) }{b}.
\end{align}
Therefore, one obtains
\begin{equation}\lb{solbound}
	\underline{u}_\varepsilon \le u(t,x)\le\overline{u}_\varepsilon,\ x\in\Gamma, t>0.  
\end{equation}

Next, let us introduce
	$$
	U(t,x):=u(t,x)-\frac{a}{b},\quad V(t,x):=v(t,x)-\frac{a}{b}.
	$$
	Then
	we have
	\begin{equation}\lb{Uequation}
U_t=U_{xx}-\chi (u_x V)_x-b u U,
	\end{equation}
	and
	\begin{equation}\lb{Vequation}
		\tau V_t=V_{xx}-V+U.
	\end{equation}
Since $U$ satisfies the Neumann--Kirchhoff vertex conditions, multiplying \eqref{Uequation} and integrating by parts as in the proof of Lemma \ref{bounds-lm1} we obtain
	\begin{align}
		\begin{split}\lb{vspteq1}
		\frac{1}{2}\frac{d}{dt}\int_\Gamma U^2dx&=-\int_\Gamma |U_x|^2dx+\chi\int_\Gamma u  U_xV_xdx-b\int_\Gamma u U^2dx\\
		&\le \frac{\chi^2}{4}\int_\Gamma u^2 |V_x|^2dx -b \int_\Gamma u U^2dx\\
		&\le \frac{\chi^2}{4}\bar u_\varepsilon ^2 \int_\Gamma |V_x|^2dx-b\underline u_\varepsilon \int_\Gamma U^2dx
		\end{split}
	\end{align}
we used \eqref{solbound} and  
\begin{equation}
\chi\int_\Gamma u  U_xV_xdx=\int_\Gamma \chi V_xu  U_x dx\leq \int_\Gamma |U_x|^2dx+\frac{\chi^2}{4}\int_\Gamma u^2 |V_x|^2dx
\end{equation}
which, in turn, follows from Young's inequality. Similarly, multiplying \eqref{Vequation} by $V$ and integrating by parts yields
	\begin{align}
		\begin{split}\lb{vspteq2}
			\frac{\tau}{2}\frac{d}{dt}\int_\Gamma V^2dx&=-\int_\Gamma |V_x|^2-\int_\Gamma V^2dx+\int_\Gamma U Vdx\\
		&\le -\int_\Gamma |V_x|^2dx-\frac{1}{2}\int_\Gamma V^2dx+\frac{1}{2}\int_\Gamma U^2dx,
		\end{split}
	\end{align}
where in the last step we used Young's inequality. Hence, combining \eqref{vspteq1}, \eqref{vspteq2} we arrive at
	\begin{align}
		\begin{split}\lb{vsp3}
			&\frac{1}{2}\frac{d}{dt}\int_\Gamma U^2dx+\frac{\tau}{2}\frac{\chi^2}{4}\overline{u}_\varepsilon^2 \int_\Gamma V^2dx\\
		&\le  -\frac{1}{2}\frac{\chi^2}{4}\overline{u}_\varepsilon^2 \int_\Gamma V^2dx -\Big(b\underline u_\varepsilon -\frac{\chi^2}{4}\overline{u}_\varepsilon^2 \Big)\int_\Gamma U^2dx. 
		\end{split}
	\end{align}
	Provided \eqref{kap} we have 
\begin{equation}
 b \underline u_\epsilon-\frac{\chi^2}{4}\overline{u}_\epsilon^2>0,
\end{equation}
which together with \eqref{vsp3} yield
\begin{equation}\lb{vsp4}
	\lim_{t\to\infty}\int_\Gamma  (U^2+V^2)dx=0.	
\end{equation}	

Let us now switch to the prove of \eqref{asympstab1}. Assume that 
	$$
	\limsup_{t\to\infty}\left(\left\|u-\frac{a}{b}\right\|_{\el{\infty}}+\left\|v-\frac{a}{v}\right\|_{\el{\infty}}\right)>0,
	$$
	Then for some $\epsilon_0>0$, $t_n\to\infty$ and $x_n\in\Gamma$, $n\in\bbN$ one has
	$$
	\left|u(t_n,x_n)-\frac{a}{b}\right|+\left|v(t_n,x_n)-\frac{a}{b}\right|\ge \epsilon_0. 
	$$
Combining this inequality with the uniform continuity of $u$ and $v$ yields  a $\delta_0>0$ such that
	$$
	\left|u(t_n,x)-\frac{a}{b}\right|+\left|v(t_n,x)-\frac{a}{b}\right|\ge \frac{\epsilon_0}{2},\ n\in\bbN, x\in\Gamma, |x-x_n|\le\delta_0.
	$$
	This implies that 
	$$
	\liminf_{n\to\infty}\int_\Gamma  (U^2(t_n,x)+V^2(t_n,x))>0,
	$$
	which contradicts \eqref{vsp4}.
\end{proof}

Finally, we prove Theorem \ref{4.4}.

\begin{proof}[Proof of Theorem \ref{4.4}]
	
	Let us define
	\begin{equation}
		\overline{u}= \limsup_{t \to \infty} \left( \sup_{x \in \Gamma} u(x,t)\right)  \quad \text{and} \quad \underline{u} =\liminf_{t \to  \infty} \left( \inf_{x \in \Gamma} u(x,t)\right),
	\end{equation}
	then for $\varepsilon>0$ there exists $t_{\varepsilon}>0$ such that
	\begin{equation}
		\underline{u} - \varepsilon \leq  \inf_{x \in \Gamma} u(x,t) \leq u(x,t) \leq \sup_{x \in \Gamma} u(x,t) \leq \overline{u}+ \varepsilon,\quad \ t\geq t_{\varepsilon}.
	\end{equation}

	Let us note that $v_{xx}(x,t)-v(x,t)+u(x,t)=0$ together with the comparison principle for elliptic equations yield 
	\begin{equation}\lb{vpeq1}
		\underline {u}_\epsilon:=\underline{u} - \epsilon \leq  v(x,t)  \leq \overline{u}+ \epsilon:=\overline u_\epsilon,\quad  x \in \Gamma, t\geq t_{\epsilon}.
	\end{equation}
	Hence, using the first equation in \eqref{parabolic-parabolic-eq} we obtain
	\begin{equation}\lb{parab1}
		u_t \leq u_{xx}-\chi u_x v_x -\chi u(\underline{u}^\epsilon-u)+u(a-bu). 
	\end{equation}  
	Consider the initial value problem for the following logistic equation 
	\begin{equation}\lb{parab2}
		\begin{cases}
			\overline{w}_t&=\chi\overline{w}(\overline{w}- \underline{u}^\epsilon) +  \overline{w}(a-b\overline{w})\  \\
			&=-(b-\chi)\overline{w}^2 +(a -\chi \underline{u}^\epsilon)\overline{w},\  t \geq t_{\epsilon},\\
			\overline{w}(t_\epsilon)&=  \underset{x \in \Gamma} \max\  u(x, t_{\varepsilon}),
		\end{cases}
	\end{equation}
	Combining \eqref{parab1}, \eqref{parab2} and the comparison principle for parabolic equations one obtains
	\begin{equation}\lb{upbnd}
		u(x,t )\leq \overline{w}(t),  x \in \Gamma, t \geq t_\varepsilon. 
	\end{equation}
	Moreover,  
	the logistic equation \eqref{parab2} yields
	\begin{equation}
		\overline{w}(t) \rightarrow \frac{(a- \chi\underline{u}^\epsilon)_+}{b-\chi} \quad \text{as} \quad t \rightarrow \infty. 
	\end{equation}
	Then employing \eqref{upbnd} one infers
	\begin{equation}
		\overline{u}\leq  \frac{{ (a- \chi(\underline{u}-\varepsilon))_+}}{b-\chi}, \quad \forall\, \varepsilon>0. 
	\end{equation}
	Hence, one has
	\begin{equation}\lb{centin2}
		\overline{u}\leq  \frac{{ (a- \chi\underline{u})_+}}{b-\chi}
	\end{equation}
	By a similar argument, using 
	\begin{equation}\lb{parab1new}
		u_t \geq  u_{xx}-\chi u_x v_x -\chi u(\overline{u}^\epsilon-u)+u(a-bu). 
	\end{equation}  
	Consider the initial value problem for the following logistic equation 
	\begin{equation}\lb{parab2new}
		\begin{cases}
			\underline{w}_t&=\chi\underline{w}(\underline{w}- \overline{u}^\varepsilon) +  \underline{w}(a-b\underline{w})\  \\
			&=-(b-\chi)\underline{w}^2 +(a -\chi \overline{u}^\varepsilon)\underline{w},\  t \geq t_{\epsilon},\\
			\overline{w}(t_\epsilon)&=  \underset{x \in \Gamma} \min\  u(x, t_{\varepsilon}),
		\end{cases}
	\end{equation}
	one obtains
	\begin{equation}\lb{centin}
		\underline{u}\geq  \frac{a- \chi\overline{u}}{b-\chi}.
	\end{equation}
	
{Note that $a-\chi\underline u>0$, for otherwise, by  \eqref{centin2}, we have
$\overline u=0$ and then $\underline u=0$. But by \eqref{centin}, we have $\underline u\ge \frac{a}{b-\chi}>0$, which yields a contradiction. Hence, one has $a-\chi\underline u>0$.} Then, 
 employing \eqref{centin}, \eqref{centin2} we obtain
	\begin{align}
		\begin{split}\lb{inaga}
			a-\chi\overline{u}+ (b-\chi)\overline{u} &\leq (b-\chi)\underline{u}+ (a-\chi\underline{u}),\\
			(b-2\chi)\overline{u} &\leq (b-2\chi)\underline{u}. 
		\end{split}
	\end{align}
	Recalling  $b > 2\chi$ and $ \overline{u} \geq \underline{u}$ we obtain $\overline{u} =\underline{u}$. This identity together with { \eqref{centin2} (resp.  \eqref{centin})   gives $\overline u\le \frac{a}{b}$ (resp. $\overline u\ge \frac{a}{b}$)}.
Hence
 $\overline{u}=\underline{u}= \frac{a}{b}$, which in turn shows
	\begin{equation}
		\lim_{t \rightarrow \infty}\left\|u(t, \cdot;u_0)-\frac{a}{b}\right\|_{L^{\infty}(\Gamma)} =0. 
	\end{equation}
	Similar inequality for $v$ follows from \eqref{vpeq1}. 
\end{proof}

\section{Numerical illustrations}\lb{numerics}
In this section we provide numerical computation of bifurcation points for the following graphs:
\begin{itemize}
		\item  the dumbbell graph, see Figure \ref{dumbbell}, whose eigenvalues $\lambda=k^2$ can be determined from the secular equation
	\begin{equation}
		\sin\frac {\ell_1 k}2\sin\frac {\ell_3 k}2\left[\left(4\sin\frac {\ell_1 k}2\sin\frac {\ell_3 k}2-\cos\frac {\ell_1 k}2\cos\frac {\ell_3 k}2\right)\sin\frac {\ell_2 k}2-2\cos\frac {\ell_2 k}2\sin\frac{(\ell_1+\ell_3)k}{2}\right]=0,
	\end{equation} 
	\item the tadpole graph, see Figure \ref{Tadpole}, whose eigenvalues $\lambda=k^2$ can be determined from the secular equation
	\begin{equation}
		2\cos(\ell_1 k)\cos(\ell_2k)-\sin(\ell_1 k)\sin(\ell_2k)=0,
	\end{equation} 
	\item the figure $8$ graph, see Figure \ref{fig8}, whose eigenvalues $\lambda=k^2$ can be determined from the secular equation
	\begin{equation}
		\sin\left(\frac {\ell_1 k}2\right)\sin\left(\frac {\ell_2 k}2\right)\sin\left(\frac {(\ell_1+\ell_2) k}2\right)=0,
	\end{equation} 
	\item  the 3-star graph, see Figure \ref{3star}, whose eigenvalues $\lambda=k^2$ can be determined from the secular equation
	\begin{align}
	\sin(\ell_1 k) \cos(\ell_2 k)\cos(\ell_3 k)+\cos(\ell_1 k) &\sin(\ell_2 k)\sin(\ell_3 k)+\\
	&+	\sin(\ell_1 k) \sin(\ell_2 k)\cos(\ell_3 k)=0.
	\end{align} 
\end{itemize}

\begin{figure}[h]
	\begin{subfigure}{.5\textwidth}
		\centering
		\includegraphics[width=1.3\linewidth]{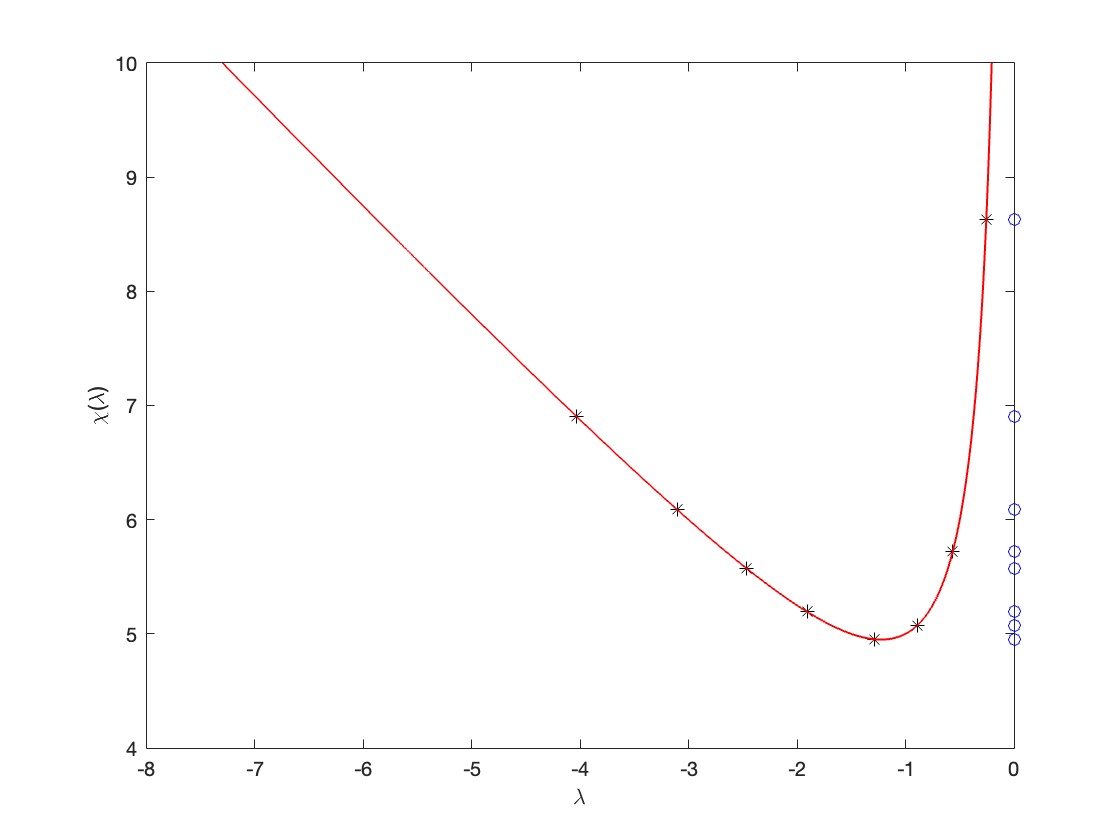}
	\end{subfigure}%
	\begin{subfigure}{.5\textwidth}
		\centering
		\includegraphics[scale=0.2]{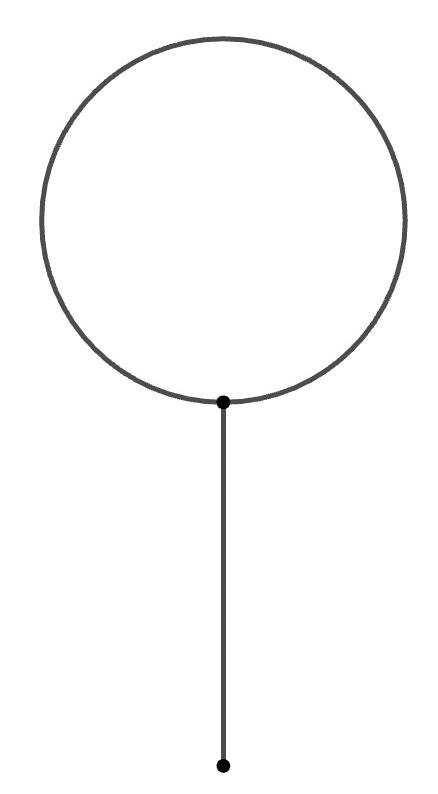}
	\end{subfigure}
	\caption{The red curve is the graph of $\chi=\chi(\lambda)$ with $a=b=1.5$; $*$ indicates the values of $\chi$ at eigenvalues of the Kirchhoff Laplacian on a {\it tadpole graph} with edge lengths $10, 5$; {\blue $\circ$} indicate bifurcation points. The first bifurcation point $\chi^*\approx4.95279$ corresponds to the $5-$th eigenvalue.}
	\label{Tadpole}
\end{figure}
\begin{figure}
	\begin{subfigure}{.5\textwidth}
		\centering
		\includegraphics[width=1.3\linewidth]{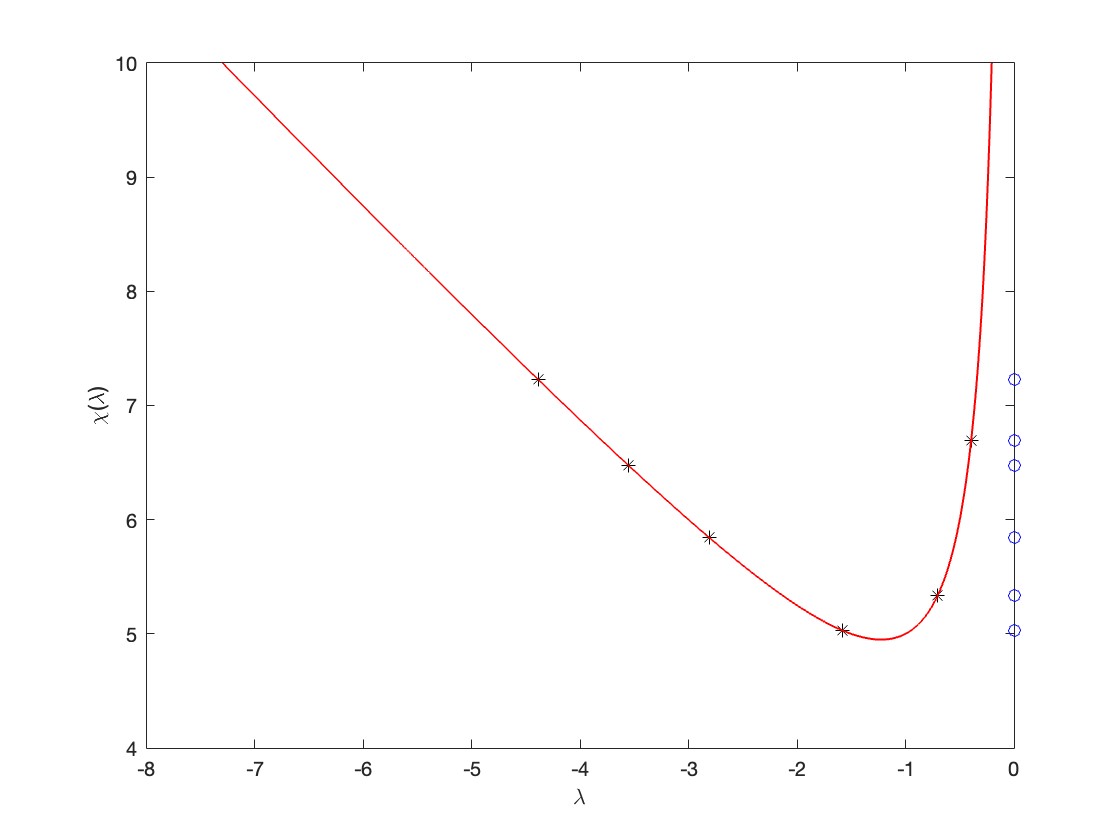}
	\end{subfigure}%
	\begin{subfigure}{.5\textwidth}
		\centering
		\includegraphics[scale=0.15]{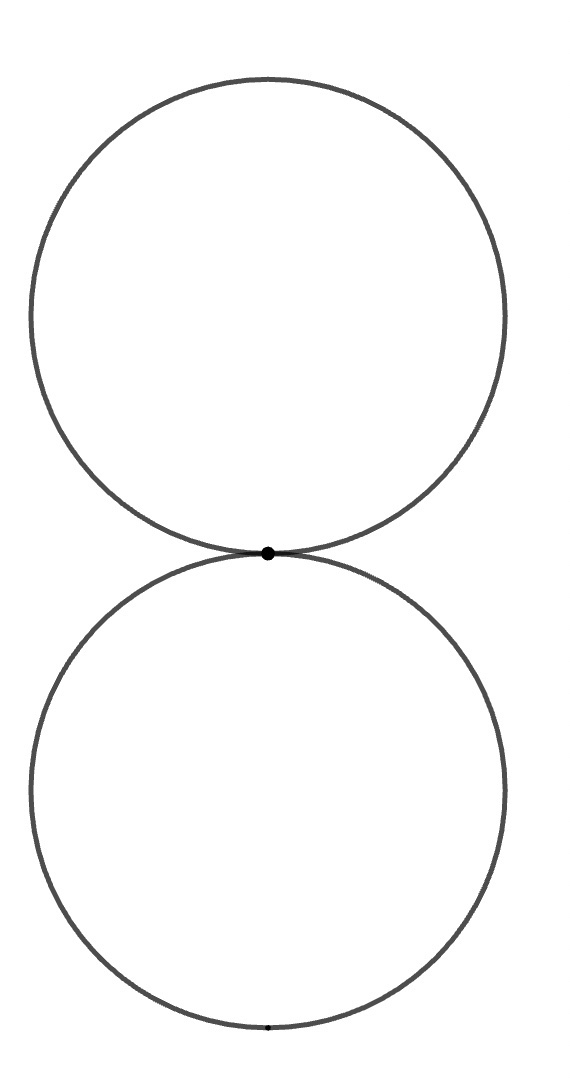}
	\end{subfigure}
	\caption{The red curve is the graph of $\chi=\chi(\lambda)$ with $a=b=1.5$; $*$ indicates the values of $\chi$ at eigenvalues of the Kirchhoff Laplacian on a {\it figure $8$ graph }with edge lengths $10, 5$; {\blue $\circ$} indicate bifurcation points. The first bifurcation point $\chi^*\approx5.01774$ corresponds to the $4-$th eigenvalue.}
	\label{fig8}
\end{figure}

\begin{figure}
	\begin{subfigure}{.5\textwidth}
		\centering
		\includegraphics[width=1.3\linewidth]{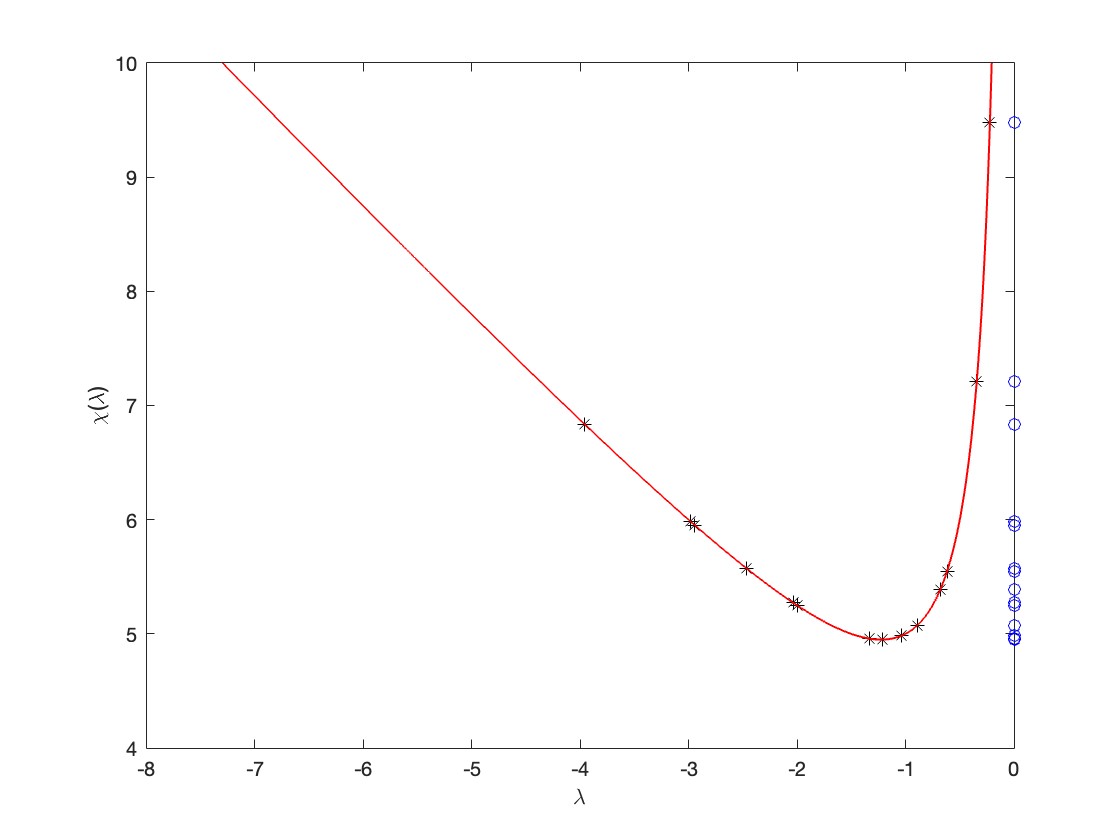}
	\end{subfigure}%
	\begin{subfigure}{.5\textwidth}
		\centering
		\includegraphics[scale=0.14]{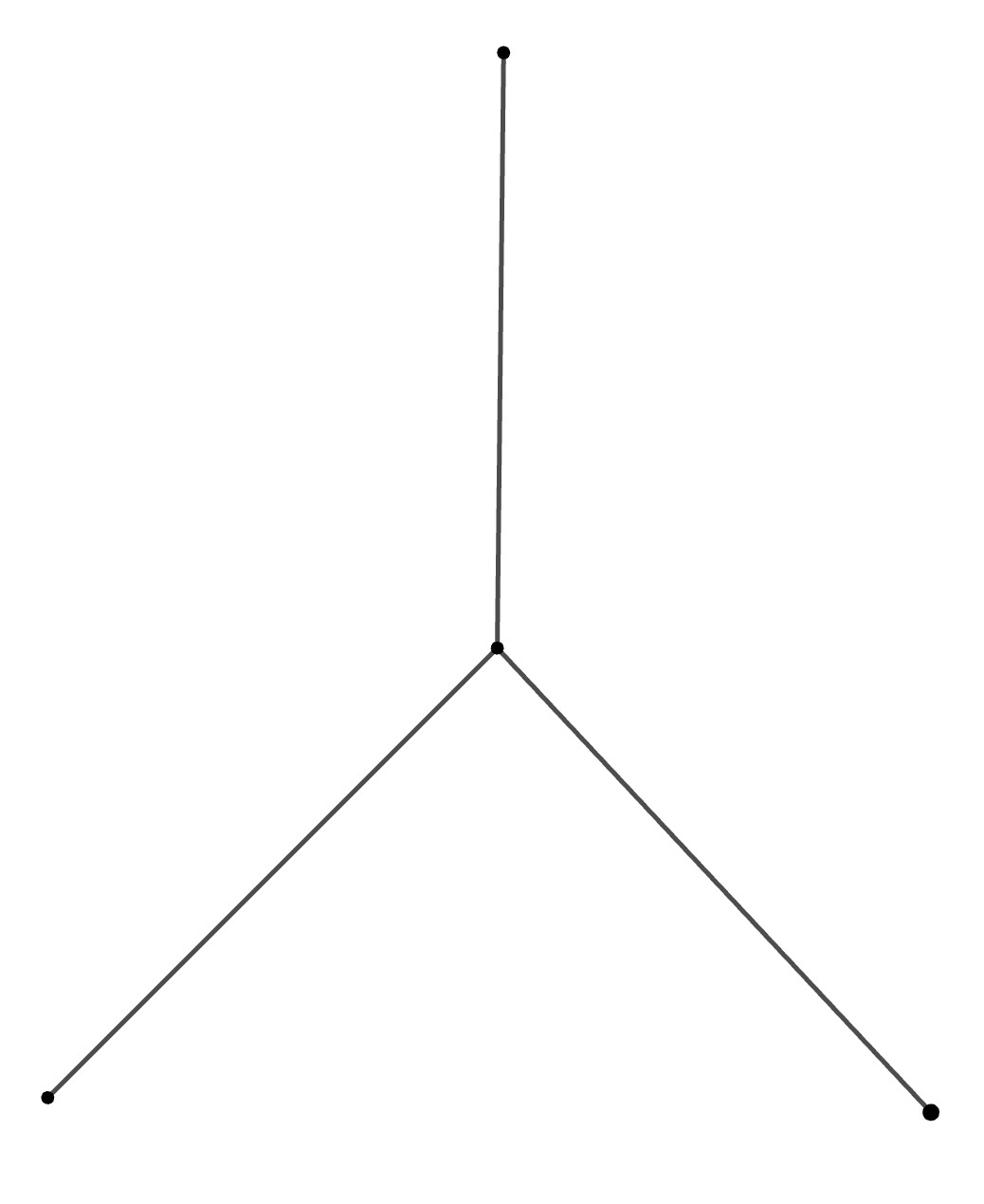}
	\end{subfigure}
	\caption{The red curve is the graph of $\chi=\chi(\lambda)$ with $a=b=1.5$; $*$ indicates the values of $\chi$ at eigenvalues of the Kirchhoff Laplacian on a {\it 3 star graph }with edge lengths $10, 5, 1$; {\blue $\circ$} indicate bifurcation points. The first bifurcation point $\chi^*\approx4.94967$ corresponds to the $8-$th eigenvalue.}
	\label{3star}
\end{figure}

\appendix
\section{Auxiliary results}\lb{functionalspaces}

In this section we record several facts about fractional power spaces $\cX^{\alpha}_p$, $\alpha\in(0,1)$, $1\leq p< \infty$ generated by the Neumann--Kirchhoff Laplacian on compact metric graphs { on $\Gamma=(\mathcal{V},\mathcal{E})$}. Let
\begin{equation}
\label{L-p-eq}
{L^p(\Gamma):=\bigoplus_{e\in\cE}L^p(e)\quad {\rm and}\quad 
\  \hatt{W}^{s,p}(\Gamma):=\bigoplus_{e\in\cE} W^{s,p}(e).}
\end{equation}  
By definition, see, e.g., \cite[Section 1.3]{Henry}, $\cX^{\alpha}_p:=\dom((I_{\elpee}-\Delta)^{\alpha})$ is equipped with the graph norm of $(I_{\elpee}-\Delta)^{\alpha}$. Throughout this paper we used bounded embeddings of $\cX^{\alpha}_p$ into various function spaces. Such embeddings are well known in the case of classical domains $\Omega\subset\bbR^n$, cf. \cite[Chapter 1]{Henry}. To the best of our knowledge, these type of embedding for metric  graphs, although expected, have not appeared in print. For completeness of exposition we present them in Theorem \ref{analytic-semigroup-thm2} below. Let us first introduce some notation.  The edge-wise direct sum of Banach spaces of functions will be denoted by  $\hatt\  $, in particular, we write
\begin{align}\lb{hats}
	\hatt C_0^{\infty}(\Gamma):= \bigoplus_{e\in\cE} C_0^{\infty}(e), &\hatt W^{k,p}(\cG):= \bigoplus_{e\in\cE}W^{k,p}(e), k\in\bbN_0,\hatt C^{\nu}(\overline{\cG}) := \bigoplus_{e\in\cE}\hatt C^{\nu}(\overline{e}),
\end{align}
where $\nu>0$ and $C^{\nu}(\overline{e})$ denotes the usual space of H\"older continuous functions defined on the closed interval $\overline{e}$ 
endowed with the norm
\begin{equation}
	\|u\|_{C^\nu(\bar e)}=\sum_{\alpha\in\bbN_0,\alpha\le [\nu]}\sup_{x\in\bar e} |u^{(\alpha)}(x)|+\sup_{x,y\in\bar e,x\not =y}
	\frac{|u^{([\nu])}(x)-u^{([\nu])}(y)|}{|x-y|^{\nu-[\nu]}}.
\end{equation}
Let us note that the edge-wise direct sums introduced in \eqref{hats} induce no vertex conditions as oppose to the  space of continuous functions on the closure $\overline{\cG}$ of the graph
\begin{equation}
	C(\overline \Gamma)=\{u\in \hatt C(\overline \Gamma)\, :\, u\,\, {\rm is \,\, continuous\,\, at\,\, the\,\, vertices\,\, of}\,\, \Gamma\}.
\end{equation}

In the following theorem $(L^p(\gamma),\hatt{W}^{2,p}(\Gamma))_{\theta,q}$ denotes the interoplation space between $L^p(\Gamma)$ and
$\hatt{W}^{2,p}(\Gamma)$ via the $K$-method, where $0<\theta<1$ and $1\le q<\infty$ (see  \cite[Section 1.3.2]{Tri} for definition). 
\begin{theorem}
	\label{analytic-semigroup-thm2} Suppose that $1\leq p<\infty$. Then one has
	\begin{itemize}
		\item[(1)] {For any $q\ge p$ and $s-\frac{1}{p}> t-\frac{1}{q}$, there holds
			\begin{equation}
				\label{sobolev-embedding-eq3-0}
				\hatt W^{s,p} (\Gamma)\hookrightarrow \hatt W^{t,q}(\Gamma),
			\end{equation}
			\begin{equation}
				\label{sobolev-embedding-eq2}
				\hatt{W}^{s,p}(\Gamma)\hookrightarrow \hatt{C}^r(\overline \Gamma)
				\quad r<s-\frac{1}{p},
			\end{equation} 
			and
			for $s\in (0,1)\setminus\{\frac{1}{2}\}$ one has
			\begin{equation}
				\label{sobolev-embedding-eq3}
				(L^p(\Gamma), \hatt{W}^{2,p}(\Gamma))_{s,p}=\hatt{W}^{2s,p}(\Gamma).
			\end{equation}
		}
		\item[(2)] One has
		\begin{equation}
			\label{sobolev-embedding-eq4}
			(\hatt{L}^p(\Gamma), X_p^\alpha)_{\theta,p}=(\hatt{L}^p(\Gamma),\mathcal{D}(A_p))_{{\alpha}\theta,p},\quad \, 0<\theta<1,
		\end{equation}
		
		\begin{equation}
			\label{sobolev-embedding-eq5-0}
			X_p^\alpha \hookrightarrow \hatt {W}^{2\alpha\theta,p}(\Gamma),\quad \, 0<\theta<1,
		\end{equation}
		
		and 
		\begin{equation}
			\label{sobolev-embedding-eq5}
			X_p^\alpha \hookrightarrow \hatt C^\nu(\overline \Gamma),\quad \,  0<\nu <2\alpha -\frac{1}{p}.
		\end{equation}
	\end{itemize}
\end{theorem}

\begin{proof}
	(1) \eqref{sobolev-embedding-eq3-0} and \eqref{sobolev-embedding-eq2} follow from \cite[Theorem 11.5]{Ama},  and
	\eqref{sobolev-embedding-eq3} follows from \cite[Theorem 11.6]{Ama}.	
	(2) First, \eqref{sobolev-embedding-eq4} follows from   \cite[(unnumbered) Theorem on page 101]{Tri}.
	
	To prove \eqref{sobolev-embedding-eq5-0},  for a given $0<\nu<2\alpha-\frac{1}{p}$, let us 
	choose $\theta\in (0,1)$ such that $\frac{\alpha}{2}\theta\not =1$ and $2 \alpha \theta-\frac{1}{p}>\nu$. Then one has
	\begin{align*}
		\cX_p^\alpha &\subset (L^p(\Gamma), \cX_p^\alpha)_{\theta,p}=(L^p(\Gamma),\mathcal{D}(A))_{{\alpha}\theta,p}\underset{\eqref{sobolev-embedding-eq4}}{\subset} ((L^p(\Gamma),\hatt{W}^{2,p}(\Gamma))_{{\alpha \theta},p}\underset{ \eqref{sobolev-embedding-eq3}}{=}\hatt{W}^{2\alpha\theta,p}(\Gamma).
	\end{align*}
	The embedding	\eqref{sobolev-embedding-eq5} follows from \eqref{sobolev-embedding-eq2} and \eqref{sobolev-embedding-eq5-0}.
\end{proof}

\begin{proposition}{\cite[Theorem 1.4.3]{MR610244}.}
	Let $\sigma>0$, $\alpha\in[0,1)$, $p\in[1, \infty)$. Then there exists $C>0$ such that for arbitrary $t>0$ and $u\in\elpee$, $v\in\cX^{\alpha}$ one has
	\begin{align}
		&\|({\sigma  I}-\Delta)^{\alpha}e^{(\Delta-{\sigma I})t}u\|_{\elpee}\leq C t^{-\alpha}e^{-\frac{ \sigma}{2} t} \|u\|_{\elpee},\lb{fracpow}\\
		&	\|(e^{(\Delta-\sigma I)t}-I)v\|_{\elpee}\leq C t^{\alpha} \|v\|_{\cX^{\alpha}_p}. \lb{fracpow2}
	\end{align}
\end{proposition}

\begin{theorem}{\cite{HWS}.}\lb{lplq estimate1}
	Let $\{e^{\Delta t}\}_{t\geq 0}$ be analytic semigroup generated by the Neumann--Kirchhoff Laplacian in $\elpee$, $1\leq p<\infty$. Let $q\in [p, \infty)$ then there exists a constant $C=C(p, q, \Gamma)>0$  such that for arbitrary $u\in C_0^\infty(\Gamma)$ and $t>0$ one has
	\begin{align}
		&\|e^{\Delta t}\partial_x u\|_{L^q(\Gamma)}\leq  Ct^{-\frac12-\frac12\left(\frac1p-\frac1q \right)}\|u\|_{\el{p}}\label{tgradlplq}.
	\end{align}
\end{theorem}


\end{document}